\documentclass[11pt,reqno]{amsart}
\usepackage{amsmath,amssymb,subfig,url}

\usepackage{tikz}
\usetikzlibrary{calc}
\usetikzlibrary{arrows}
\usetikzlibrary{patterns}
\usetikzlibrary{through}
\usetikzlibrary{plotmarks}
\usepackage{subfig}
\usepackage{afterpage}

\newtheorem{theorem}{Theorem}[section]
\newtheorem{observation}[theorem]{Numerical Observation}

\newtheorem{lemma}[theorem]{Lemma}
\newtheorem{proposition}[theorem]{Proposition}

\newtheorem{conjecture}[theorem]{Conjecture}

\newtheorem*{remark}{Remark}

\newtheorem*{definition}{Definition}

\numberwithin{equation}{section}

\newcommand{\e}{\varepsilon}
\newcommand{\R}{{\mathbb R}}

\begin{document}

\title[]{Dirichlet eigenvalue sums on triangles are minimal for equilaterals}

\author[]{R. S. Laugesen and B. A. Siudeja}
\address{Department of Mathematics, University of Illinois, Urbana,
IL 61801, U.S.A.} \email{Laugesen\@@illinois.edu, Siudeja\@@illinois.edu}
\date{\today}


\begin{abstract}
Among all triangles of given diameter, the equilateral triangle is shown to minimize the sum of the first $n$ eigenvalues of the Dirichlet Laplacian, for each $n \geq 1$. In addition, the first, second and third eigenvalues are each proved to be minimal for the equilateral triangle.

The disk is conjectured to be the minimizer among general domains.
\end{abstract}

\maketitle

\section{\bf Results and Conjectures}

This paper establishes geometrically sharp lower bounds for the Dirichlet eigenvalues of the Laplacian on triangular domains. These eigenvalues are important in solving the wave, diffusion and Schr\"{o}dinger equations. The eigenfunctions satisfy $-\Delta u = \lambda u$ with boundary condition $u = 0$, and the eigenvalues satisfy
\[
0 < \lambda_1 < \lambda_2 \leq \lambda_3 \leq \dots \to \infty .
\]

We start with a sharp bound on eigenvalue sums. Notice we normalize the eigenvalues to be scale invariant, by multiplying with the square of the diameter $D$ of the domain.
\begin{theorem} \label{th:trace}
Among all triangular domains of given diameter, the equilateral triangle minimizes the sum of the first $n$ eigenvalues of the Dirichlet Laplacian, for each $n$.

That is, if $T$ is a triangular domain, $E$ is equilateral, and $n \geq 1$, then
\[
(\lambda_1 + \cdots + \lambda_n) D^2 \Big|_T \geq (\lambda_1 + \cdots + \lambda_n) D^2 \Big|_E
\]
with equality if and only if $T$ is equilateral.
\end{theorem}
The theorem is proved in Section \ref{sectraces}.

The case $n=1$ for the fundamental tone was known already, by combining P\'{o}lya and Szeg\H{o}'s result \cite[Note A]{PS51} that $\lambda_1 A$ is minimal for the equilateral triangle with the elementary isodiametric inequality that $D^2/A$ is minimal among triangles for the equilateral triangle. For $n = 2$ this argument fails, because $(\lambda_1+\lambda_2)A$ is \emph{not} minimal for the equilateral triangle (see Figure~\ref{figLA}b). Thus the extension from $n=1$ to $n \geq 1$ in Theorem~\ref{th:trace} is possible because we weaken the geometric normalization from area to diameter.

The eigenvalues of the equilateral triangle are known explicitly. We recall the formula in Appendix~\ref{equilateral}.

Lower bounds on eigenvalues, as in Theorem~\ref{th:trace}, are generally difficult to establish because eigenvalues are characterized by \emph{minimization} of a Rayleigh quotient. The best known lower bound on the first eigenvalue of a general domain is the Rayleigh--Faber--Krahn theorem that the fundamental tone is minimal for the disk, among domains of given area; in other words, the scale invariant quantity $\lambda_1 A$ is minimal for the disk. The standard proof is by rearrangement. Rearrangement methods cannot be applied for higher eigenvalues, $n>1$, because the higher modes change sign. Our method is different. As we explain in Section \ref{sectraces}, this new
\[
\text{Method of the Unknown Trial Function}
\]
involves transplanting the (unknown) eigenfunctions of an arbitrary triangle to yield trial functions for the (known) eigenvalues of certain equilateral and right triangles. Transplantation distorts the derivatives in the Rayleigh quotient, but by ``interpolating'' the various transplantations, we can eliminate the distortions and arrive at sharp lower bounds on the eigenvalues of the arbitrary triangle.

\begin{remark} \rm
Theorem~\ref{th:trace} gives a geometrically sharp result on all eigenvalue sums. It is the first sharp lower bound of this kind, so far as we are aware.

One might contrast the Theorem with the lower bound of Berezin \cite{B72} and Li and Yau \cite{LY83} that $(\lambda_1 + \cdots + \lambda_n) A > 2\pi n^2$. This estimate is not geometrically sharp because there is no domain known on which equality holds, when $n$ is fixed. The Berezin--Li--Yau inequality is instead asymptotically sharp for each domain as $n \to \infty$, by the Weyl asymptotic $\lambda_n A \sim 4\pi n$.
\end{remark}

Next we minimize the eigenvalues individually.
\begin{theorem} \label{th:secondtone}
For each triangular domain one has
\[
\lambda_1 D^2 \geq \frac{3 \cdot 16 \pi^2}{9}  , \qquad \lambda_2 D^2 \geq \frac{7 \cdot 16 \pi^2}{9}   \qquad \text{and} \qquad \lambda_3 D^2 \geq \frac{7 \cdot 16 \pi^2}{9}  .
\]
In each inequality, equality holds if and only if $T$ is equilateral.
\end{theorem}
The difficult result in the theorem is the inequality for $\lambda_2$. We prove it by writing $\lambda_2=(\lambda_1+\lambda_2)-\lambda_1$ and then combining a new lower bound on $\lambda_1 + \lambda_2$ with a known upper bound on $\lambda_1$. See Section \ref{sec:secondtone_proof}. Note the inequality for $\lambda_3$ in the theorem follows directly from the one for $\lambda_2$, and the inequality for the first eigenvalue is just the case $n=1$ of Theorem~\ref{th:trace}.

It is conceivable that \emph{every} eigenvalue might be minimal for the equilateral triangle, which would extend Theorem~\ref{th:secondtone} to all $n$.
\begin{conjecture} \label{co:alltones}
The equilateral triangle minimizes each eigenvalue of the Dirichlet Laplacian, among all triangles of given diameter. That is, $\lambda_n D^2$ is minimal when the triangle is equilateral, for each $n \geq 1$.
\end{conjecture}
Among rectangles the analogous conjecture is false: the square is not always the minimizer, by the following explicit calculation. The rectangle with sides of length $\cos \phi$ and $\sin \phi$ has diameter $D=1$, and eigenvalues $\lambda_1 = \pi^2(\sec^2 \phi + \csc^2 \phi)$ and $\lambda_2 = \pi^2(2^2\sec^2 \phi + \csc^2 \phi)$ when $0<\phi \leq \pi/4$. One computes that $\lambda_2$ is minimal for some $\phi < \pi/4$, in other words, for some non-square rectangle, and similarly for the sum $\lambda_1+\lambda_2$. Thus the theorems in this paper do not adapt to rectangles.

Nonetheless, our work on triangles tends to support a conjecture for general domains.
\begin{conjecture} \label{co:trace}
The disk minimizes the sum of the first $n \geq 1$ eigenvalues of the Dirichlet Laplacian, among all plane domains of given diameter. That is, $(\lambda_1 + \cdots + \lambda_n) D^2$ is minimal when the domain is a disk.
\end{conjecture}
Perhaps even the individual eigenvalues are minimal for the disk.
\begin{conjecture} \label{co:secondtone}
The disk minimizes each eigenvalue of the Dirichlet Laplacian, among all plane domains of given diameter. That is, $\lambda_n D^2$ is minimal when the domain is a disk, for each $n$.
\end{conjecture}
For $n=1$, the last two Conjectures are true by the Faber--Krahn theorem that $\lambda_1 A$ is minimal for the disk together with the isodiametric theorem that $D^2/A$ is minimal for the disk.

For $n=2$, the minimality of $\lambda_2 D^2$ for the disk was suggested already by Bucur, Buttazzo and Henrot \cite{BBH09}, who conjectured that
\begin{equation} \label{BBHconj}
\lambda_2 D^2 \geq 4 j_{1,1}^2
\end{equation}
for all bounded plane domains (with equality for the disk). Our paper is motivated by their conjecture. Notice Theorem~\ref{th:secondtone} improves considerably on the conjectured \eqref{BBHconj}, for triangles, because $7 \cdot 16 \pi^2/9 \simeq 123$ is greater than $4j_{1,1}^2 \simeq 59$.

For all $n$, Conjectures~\ref{co:trace} and \ref{co:secondtone} are true in the sub-class of ellipses, because an ellipse of diameter $D$ lies inside a disk of the same diameter, and the eigenvalues of the disk are lower by domain monotonicity.

Summing up, then, the main contributions of the paper are that it develops a new method for proving geometrically sharp lower bounds on eigenvalues, that it handles eigenvalue sums of arbitrary length, and that it supports new and existing conjectures on general domains by investigating the interesting class of triangular domains.

\medskip
We contribute also some numerical observations about symmetry properties of second eigenfunctions of isosceles triangles. Given an isosceles triangle with line of symmetry ${\mathcal L}$, we call an eigenvalue \emph{(anti)symmetric} if its corresponding eigenfunction is (anti)symmetric in ${\mathcal L}$. For example, the fundamental tone of an isosceles triangle is symmetric, since the fundamental mode is symmetric (by uniqueness).
\begin{definition}\rm
The \emph{aperture} of an isosceles triangle is the angle between
its two equal sides. Call a triangle \emph{subequilateral} if it is isosceles with
aperture less than $\pi/3$, and \emph{superequilateral} if it is
isosceles with aperture greater than $\pi/3$.
\end{definition}
\begin{observation} \label{subeq} \ 

\noindent The second mode of a subequilateral triangle is symmetric 
 (Figure~\ref{symmfig}(a)).

\noindent The second mode of a superequilateral triangle is antisymmetric 
 (Figure~\ref{symmfig}(b)).
\end{observation}
The result is plausible because the oscillation of the second mode should take place in the ``long'' direction of the triangle.
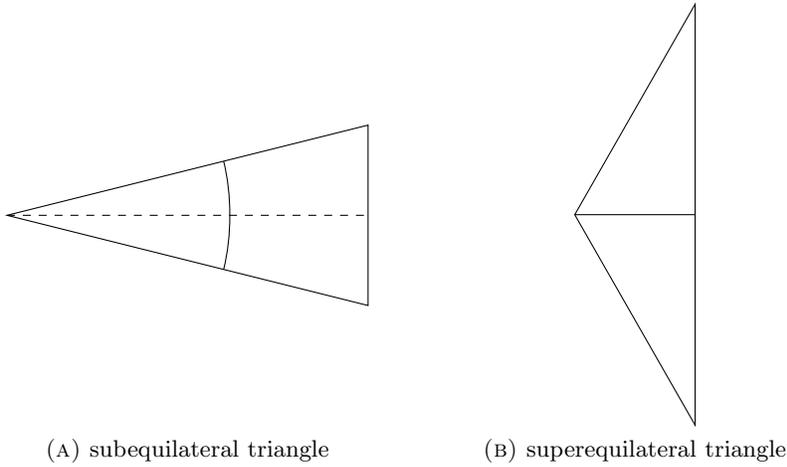
\begin{figure}[t]
  \begin{center}
    \hspace*{\fill}
    \subfloat[subequilateral triangle]{
\begin{tikzpicture}[scale=0.4]
      \path (0,-7) -- (12,7);
      \draw (0,0) -- (12,-3) -- (12,3) -- cycle;
      \clip (0,0) -- (12,-3) -- (12,3) -- cycle;
      \draw[dashed] (0,0) -- (12,0);
      \draw (7.2,-3*7.2/12) .. controls +(3/12*1.1,1.1) and +(3/12*1.1,-1.1) .. (7.2,3*7.2/12);
    \end{tikzpicture}
    }
    \hspace*{\fill}
    \subfloat[superequilateral triangle]{
    \begin{tikzpicture}[scale=0.4]
      \path (0,-7) -- (12,7);
      \draw (4,0) -- (8,7) -- (8,-7) -- cycle;
      \draw (4,0) -- (8,0);
    \end{tikzpicture}
    }
    \hspace*{\fill}
  \end{center}
  \caption{Nodal curve (solid) for the second eigenfunction of an isosceles triangle. The mode satisfies a Dirichlet condition on each solid curve, and a Neumann condition on the dashed line.} \label{symmfig}
\end{figure}
Observation~\ref{subeq} is based on a numerical plot of the lowest symmetric and antisymmetric eigenvalues, in Figure~\ref{figl} later in the paper. We have not found a rigorous proof.

Further properties of the low eigenvalues of isosceles triangles, such as monotonicity properties with respect to the aperture, will be proved in Section~\ref{sec:isos}.

\medskip
The rest of the paper is organized as follows. The next section surveys known estimates on low eigenvalues, especially for the fundamental tone and the spectral gap. Then Section \ref{sectraces} begins the proof of Theorem \ref{th:trace} on eigenvalue sums, and introduces the Method of the Unknown Trial Function. The behavior of eigenvalue sums under linear transformation of the domain is studied in Section~\ref{isec3}. High eigenvalues of equilateral triangles are estimated in Section~\ref{seccomp}. Finally, Section \ref{sec:secondtone_proof} proves Theorem \ref{th:secondtone} on the second eigenvalue.

\section{\bf Literature and related results} \label{literature}

This paper proves lower bounds that are geometrically sharp, on eigenvalue sums of arbitrary length. The only similar results we know are the \emph{upper} bounds in our companion paper \cite{LS10b}, where the triangles are normalized by the ratio $\text{(area)}^3/\text{(moment of inertia)}$, rather than by $\text{(diameter)}^2$, and where equilateral triangles are shown to maximize (rather than minimize) the eigenvalue sums.

Results that are \emph{asymptotically} sharp have received more attention. The P\'{o}lya conjecture $\lambda_n A > 4\pi n$ asserts that the Weyl asymptotic estimate is a strict lower bound on each Dirichlet eigenvalue. It has been proved for tiling domains \cite{P61}, but remains open for general domains for $n \geq 3$; results on product domains are due to Laptev \cite{L97}, and counterexamples under constant magnetic field have been found by Frank, Loss and Weidl \cite{FLW09}. Note that the Berezin--Li--Yau inequality mentioned in the previous section is a ``summed'' version of P\'{o}lya's conjecture, and that it was generalized to homogeneous spaces by Strichartz \cite{S96}.

Eigenvalues of triangular domains have been studied intensively \cite{AF06,AF08,F06,F07,FS10,LS09a,LS10,S07,S10}. The most notable lower bounds on the first Dirichlet eigenvalue are P\'{o}lya and Szeg\H{o}'s inequality \cite[Note A]{PS51} of Faber--Krahn type that
\[
\lambda_1 A \geq \frac{4\pi^2}{\sqrt{3}}
\]
with equality for the equilateral triangle, and Makai's result \cite{M62} that
\[
\lambda_1 \frac{A^2}{L^2} \geq \frac{\pi^2}{16}
\]
with equality for the degenerate acute isosceles triangle. This last inequality can be interpreted in terms of inradius normalization, since the inradius is proportional to $A/L$, for triangles.

Another interesting functional is the spectral gap $\lambda_2-\lambda_1$, which was conjectured by van den Berg \cite{vdB} to be minimal for the degenerate rectangular box, among all convex domains of diameter $D$, in other words, that
\[
(\lambda_2-\lambda_1) D^2 > 3\pi^2 .
\]
A proof was presented recently by Andrews and Clutterbuck \cite{AC10}. Among triangles, the gap minimizer is conjectured by Antunes and Freitas \cite{AF08} to be the equilateral triangle (and not a degenerate triangle). Lu and Rowlett \cite{LR10} have announced that a proof will appear in a forthcoming paper. An alternative approach for triangles might be to modify our proof of Theorem~\ref{th:secondtone}, using that $\lambda_2-\lambda_1=(\lambda_1+\lambda_2)-2\lambda_1$. We have not succeeded with this approach, unfortunately.

We investigated Neumann eigenvalues (rather than Dirichlet) in two recent works \cite{LS09a,LS10}. The first of those papers maximized the low Neumann eigenvalues of triangles, under perimeter or area normalization, with the maximizer being
equilateral. The second paper minimized the second Neumann eigenvalue (the first nonzero eigenvalue) under diameter or perimeter normalization, with the minimizer being degenerate acute isosceles; a consequence is a sharp Poincar\'{e} inequality for triangles.

For broad surveys of isoperimetric eigenvalue inequalities, see the paper by Ashbaugh \cite{A99}, and the monographs of Bandle \cite{B79}, Henrot \cite{He06}, Kesavan
\cite{K06} and P\'{o}lya--Szeg\H{o} \cite{PS51}.

\section{\bf Proof of Theorem~\ref{th:trace}: the Method of the Unknown Trial Function}\label{sectraces}

We begin with some notation. Write
\[
\Lambda_n = \lambda_1 + \cdots + \lambda_n
\]
for the sum of the first $n$ eigenvalues. Define
\[
T(a,b) = \text{triangle having vertices at $(-1,0), (1,0)$ and $(a,b)$,}
\]
where $a \in \R$ and $b>0$. Choosing $a=0$ gives an isosceles triangle; it is subequilateral if in addition $b>\sqrt{3}$.

The theorem will be proved in three steps.

\medskip
\noindent Step 1 --- Reduction to isosceles triangles.

Every triangle is contained in a subequilateral triangle with the same diameter, simply by extending the second-longest side to be as long as the longest side. The Dirichlet eigenvalues of this subequilateral triangle are lower than those of the original triangle, by domain monotonicity. By dilating, translating and rotating, we may suppose the subequilateral triangle is $T(0,b)$ for some $b > \sqrt{3}$. Thus it suffices to prove that the theorem holds with strict inequality for $T=T(0,b)$.

\medskip
\noindent Step 2 --- Method of the Unknown Trial Function.

Assume $b>\sqrt{3}$. Define three special triangles:
\begin{align*}
E & = T(0,\sqrt{3}) = \text{equilateral triangle} , \\
F_\pm & = T(\pm 1,2\sqrt{3}) = \text{30-60-90${}^\circ$ right triangle.}
\end{align*}
See Figure~\ref{fig:lineartrans}.

\begin{proposition} \label{unknown}
If $b > \sqrt{3}$ then
\[
\left. \Lambda_n D^2 \right|_{T(0,b)} > \min \left\{ \left. \Lambda_n D^2 \right|_E , \, \frac{6}{11} \! \left. \Lambda_n D^2 \right|_{F_\pm} \right\}
\]
for each $n \geq 1$.
\end{proposition}
We give the proof in the next section. The idea is to estimate the eigenvalues of the subequilateral triangle $T(0,b)$ by linearly transplanting its eigenfunctions to provide trial functions on the equilateral and right triangles, whose eigenvalues are known explicitly. Figure~\ref{fig:lineartrans} indicates the linear maps. The figure suggests that we may regard the subequilateral triangle as ``interpolating'' between the equilateral and right triangles, in some sense.

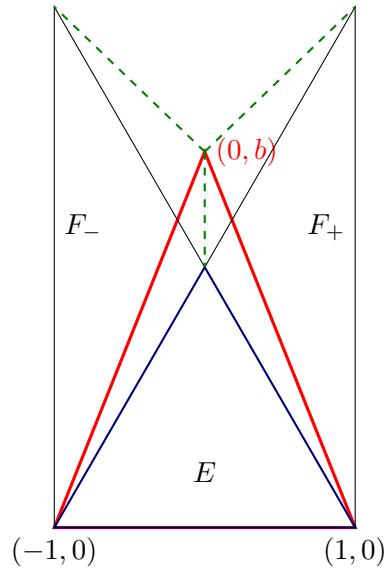
\begin{figure}[ht]
  \begin{center}
	\begin{tikzpicture}[scale=2]
	  \draw[very thick,red] (-1,0) -- (1,0) -- (0,2.5) node [right] {$(0,b)$} -- cycle;
	  \draw[thick, blue] (-1,0) -- (1,0) -- (0,{sqrt(3)}) -- cycle;
	  \draw (0,1/2) node [below] {$E$} ;
	  \draw (-1,0) node [below] {$(-1,0)$} -- (1,0) -- (-1,{2*sqrt(3)}) -- cycle;
	  \draw (-1,2) node [right] {$F_-$} ;
	  \draw (1,0) node [below] {$(1,0)$} -- (-1,0) -- (1,{2*sqrt(3)}) -- cycle;
	  \draw[green!50!black,dashed,thick] (0,2.5) -- (-1,{2*sqrt(3)}) ;
	  \draw (1,2) node [left] {$F_+$} ;
	  \draw[green!50!black,dashed,thick] (0,2.5) -- (1,{2*sqrt(3)}) ;
	  \draw[green!50!black,dashed,thick] (0,2.5) -- (0,{sqrt(3)}) ;
	\end{tikzpicture}
  \end{center}
  \caption{Linear maps to the subequilateral triangle $T(0,b)$, from the equilateral triangle $E$ and right triangles $F_\pm$.}
  \label{fig:lineartrans}
\end{figure}

We call this approach to proving lower bounds the \emph{Method of the Unknown Trial Function}, because the eigenfunctions of the subequilateral triangle are unknown and it is these eigenfunctions that generate the trial functions for the (known) eigenvalues of the equilateral and right triangles. This method contrasts with the usual approach to proving an \emph{upper} bound, where one puts a known trial function into the Rayleigh quotient in order to estimate the unknown eigenvalue.

\medskip
\noindent Step 3 --- Comparing the right and equilateral triangles.

To complete the proof, we combine Proposition~\ref{unknown} with:
\begin{lemma} \label{compequilateral}
$\frac{6}{11} \! \left. \Lambda_n D^2 \right|_{F_\pm} > \left. \Lambda_n D^2 \right|_E$ for each $n \geq 1$.
\end{lemma}
The proof is in Section~\ref{seccomp}. Note the lemma is plausible for large $n$, because $\lambda_n D^2 \sim 4\pi n D^2/A$ by the Weyl asymptotic and $D^2/A \big\rvert_{F_\pm}=2D^2/A \big\rvert_E$ since the diameter and area of $F_\pm$ are twice as large as for $E$.

\section{\bf Linear transformation of eigenfunctions: proof of Proposition~\ref{unknown}} \label{isec3}

The Rayleigh Principle says the fundamental tone of a bounded plane domain $\Omega$ is
\[
\lambda_1 = \min_{v \in H^1_0(\Omega) \setminus \{ 0 \} } R[v]
\]
where
\[
R[v] = \frac{\int_\Omega |\nabla v|^2 \, dA}{\int_\Omega v^2 \, dA}
\]
is the \emph{Rayleigh quotient}. Similarly, the sum of the first $n$ eigenvalues is
\begin{align*}
\Lambda_n & = \lambda_1 + \cdots + \lambda_n \\
& = \min \Big\{ R[v_1] + \cdots + R[v_n] : v_j \in H^1_0(\Omega) \setminus \{ 0 \} , \text{$\langle v_j , v_k \rangle_{L^2}=0$ when $j \neq k$} \Big\} .
\end{align*}

We will transplant eigenfunctions from one triangle to another, to build trial functions for the Rayleigh quotients.

Recall the triangle $T(a,b)$ having vertices at $(-1,0), (1,0)$ and $(a,b)$, where $b>0$. Denote the eigenvalues of this triangle by $\lambda_j(a,b)$, and write $u_j$ for the corresponding $L^2$-orthonormal eigenfunctions. Define $\Lambda_n(a,b) = \sum_{j=1}^n \lambda_j(a,b)$.
\begin{lemma}[Linear transformation and eigenvalue sums] \label{lemtrace}
Let $a,c \in \R$ and $b,d>0$. Take $C > 0$ and $n \geq 1$.

The inequality
\[
\Lambda_n(a,b) > C \Lambda_n(c,d)
\]
holds if
\begin{equation} \label{ineq}
\frac{1}{d^2} \Big[ \big( (a-c)^2+d^2 \big) (1-\gamma_n) + 2b(a-c)\delta_n + b^2\gamma_n \Big] < \frac{1}{C},
\end{equation}
where
\[
\gamma_n = \frac{\sum_{j=1}^n \int_{T(a,b)} u_{j,y}^2 \, dA}{\sum_{j=1}^n \int_{T(a,b)} |\nabla u_j|^2 \, dA} \qquad \text{and} \qquad
\delta_n = \frac{\sum_{j=1}^n \int_{T(a,b)} u_{j,x} u_{j,y} \, dA}{\sum_{j=1}^n \int_{T(a,b)} |\nabla u_j|^2 \, dA} .
\]
\end{lemma}

\begin{proof}[Proof of Lemma~\ref{lemtrace}]
Let $\tau$ be the linear transformation fixing the vertices $(-1,0)$ and $(1,0)$ and mapping the point $(c,d)$ to $(a,b)$; that is,
\[
  \tau(x,y)=\left( x+\frac{a-c}{d} y,\frac{b}{d} y \right) .
\]
This transformation maps the triangle $T(c,d)$ to $T(a,b)$.

The functions $u_j \circ \tau$ are $L^2$-orthogonal on $T(c,d)$, and so the Rayleigh principle gives
\begin{align*}
\sum_{j=1}^n \lambda_j(c,d) & \leq \sum_{j=1}^n R[u_j \circ \tau] = \sum_{j=1}^n \frac{\int_{T(c,d)} |\nabla(u_j \circ \tau)|^2 \, dA}{\int_{T(c,d)} (u_j \circ \tau)^2 \, dA} \\
& = \sum_{j=1}^n \frac{ \int_{T(a,b)} d^{-2} \big[ \big( (a-c)^2+d^2 \big) u_{j,x}^2 + 2b(a-c) u_{j,x} u_{j,y}+b^2 u_{j,y}^2 \big] \, dA}{\int_{T(a,b)} u_j^2 \, dA}
\end{align*}
by the chain rule and a change of variable back to $T(a,b)$. The denominator equals $1$ because the eigenfunctions are normalized, and so
\begin{align*}
\sum_{j=1}^n \lambda_j(c,d) & \leq d^{-2} \big[ \big( (a-c)^2+d^2 \big)(1-\gamma_n)  + 2b(a-c)\delta_n + b^2\gamma_n \big] \sum_{j=1}^n \int_{T(a,b)} |\nabla u_j|^2 \, dA \\
& < \frac{1}{C} \sum_{j=1}^n \lambda_j(a,b)
\end{align*}
by assumption \eqref{ineq}.
\end{proof}

\begin{remark}
\rm It is not important that the domains be triangular, in this method. We simply need one domain to be the image of the other under a linear transformation. For example, similar results hold for parallelograms, and for elliptical domains.
\end{remark}

\begin{proof}[Proof of Proposition~\ref{unknown}]
The equilateral triangle $E = T(0,\sqrt{3})$ has diameter $2$, and the subequilateral triangle $T(0,b)$ has diameter $\sqrt{1+b^2}$. Observe that the desired inequality
\[
   \left. \Lambda_n D^2 \right|_{T(0,b)} = \Lambda_n(0,b) (1+b^2) > \Lambda_n(0,\sqrt{3}) 2^2 = \left. \Lambda_n D^2 \right|_E
\]
holds by Lemma~\ref{lemtrace} with $a=c=0, d=\sqrt{3}$ and $C=4/(1+b^2)$, if
\[
(1-\gamma_n) + \frac{1}{3} b^2 \gamma_n < \frac{1+b^2}{4} .
\]
This last inequality is equivalent to $\gamma_n < 3/4$. Thus if $\gamma_n<3/4$ then the Proposition is proved.

Suppose $\gamma_n \geq 3/4$, and remember $\gamma_n \leq 1$ by definition. Recall the right triangles $F_\pm = T(\pm 1,2\sqrt{3})$, which have diameter $4$. Observe that
\begin{equation} \label{eq:again}
\left. \Lambda_n D^2 \right|_{T(0,b)} = \Lambda_n(0,b) (1+b^2) > \frac{6}{11} \Lambda_n(\pm 1,2\sqrt{3}) 4^2 = \frac{6}{11} \left. \Lambda_n D^2 \right|_{F_\pm}
\end{equation}
holds by Lemma~\ref{lemtrace} with $a=0, c=\pm 1, d=2\sqrt{3}$ and $C=\frac{6}{11} \frac{4^2}{1+b^2}$, if
\[
\frac{1}{12} \big[ 13(1-\gamma_n) \mp 2b\delta_n + b^2 \gamma_n \big] < \frac{11}{6} \, \frac{1+b^2}{4^2} .
\]
The eigenvalues of $F_+$ and $F_-$ are the same, and so we need only establish this last inequality for ``$+$'' or for ``$-$''. It holds for at least one of these sign choices if
\[
\frac{1}{12} \big[ 13(1-\gamma_n) + b^2\gamma_n \big] < \frac{11}{6} \, \frac{1+b^2}{4^2} ,
\]
which is equivalent to
\[
b^2 + \frac{50}{11-8\gamma_n} > 13 .
\]
This inequality is true because $b > \sqrt{3}$ and  $\gamma_n \geq 3/4$. (Equality holds when $b = \sqrt{3}$ and  $\gamma_n = 3/4$.) Hence \eqref{eq:again} holds, which completes the proof.
\end{proof}

\begin{remark} \rm
The proof avoids estimating $\gamma_n$. Recall that this quantity measures the fraction of the Dirichlet energy of the eigenfunctions $u_1,\ldots,u_n$ that is provided by their $y$-derivatives. Since these eigenfunctions are not known to us, it is important that the proof be able to handle any value of $\gamma_n$.
\end{remark}

\section{\bf Comparison of eigenvalue sums: proof of Lemma~\ref{compequilateral}} \label{seccomp}

To prove Lemma~\ref{compequilateral}, we need the following Weyl-type bounds on the eigenvalues of the equilateral triangle with sidelength $1$.

Call that triangle $E_1$, and define its counting function
\[
N(\lambda) = \# \{ j \geq 1 : \lambda_j(E_1) < \lambda \} .
\]
\begin{lemma} \label{counting}
The counting function satisfies
\[
\frac{\sqrt{3}}{16\pi} \lambda -\frac{\sqrt{3}}{4\pi}\sqrt{\lambda} + \frac{1}{2} > N(\lambda) > \frac{\sqrt{3}}{16\pi} \lambda - \frac{(6-\sqrt{3})}{4\pi} \sqrt{\lambda} - \frac{1}{2} , \qquad \forall \, \lambda > 48\pi^2 .
\]
Hence for all $j \geq 17$,
\begin{align*}
& \frac{16\pi}{\sqrt{3}}\Big( j-\frac{1}{2} \Big) + 8 \sqrt{\frac{4\pi}{\sqrt{3}} \Big( j-\frac{1}{2} \Big) + 1} + 8 \\
& \leq \lambda_j(E_1) \\
& < \frac{16\pi}{\sqrt{3}}\Big( j+\frac{1}{2} \Big) + \frac{4}{\sqrt{3}}(6-\sqrt{3}) \sqrt{ \frac{16\pi}{\sqrt{3}} \Big( j+\frac{1}{2} \Big) + 4(13-4\sqrt{3})} + 8(13-4\sqrt{3}) \\
& < (29.03)j+9.9\sqrt{29.03j+39}+64 .
\end{align*}
\end{lemma}
\begin{proof}[Proof of Lemma~\ref{counting}]
The eigenvalues of the equilateral triangle (see Appendix~\ref{equilateral}) are
\[
  \sigma_{m,n} = \frac{16\pi^2}{9}(m^2+mn+n^2), \qquad m,n \geq 1 .
\]
The number of eigenvalues less than $\lambda$ is precisely the size of the set
\[
Q = \{ (m,n) : m,n \geq 1, m^2+mn+n^2 < R^2 \} ,
\]
where
\[
R^2= \frac{9}{16\pi^2} \lambda , \qquad \text{that is,} \qquad R= \frac{3}{4\pi} \sqrt{\lambda}.
\]
Thus to estimate $N(\lambda)$, our task is to count the lattice points in the first quadrant that lie inside the ellipse
\[
x^2+xy+y^2 = R^2 .
\]
The portion of the ellipse contained in the first quadrant has area $\pi R^2/3\sqrt{3}$, by an elementary calculation. This area exceeds the number of lattice points in $Q$, because each lattice point $(m,n)$ corresponds to a closed unit square with upper right vertex $(m,n)$ and with the whole square lying within the ellipse. Hence
\begin{equation} \label{weyltype}
N(\lambda) <  \frac{\pi R^2}{3\sqrt{3}}  = \frac{\sqrt{3}}{16\pi} \lambda .
\end{equation}
We aim to improve this Weyl-type estimate by subtracting a term of order $\sqrt{\lambda}$. Our proof will assume $R > 3\sqrt{3}$, which explains the restriction to $\lambda > 48\pi^2$ in the lemma.

For each $m$ with $1 \leq m < R$, define $n(m)$ to be the largest integer $n$ for which $(m,n) \in Q$, and define $n(m)=0$ when there is no such $n$. The function $n(\cdot)$ is decreasing, since the ellipse has a decreasing graph. (Figure \ref{figell} shows the part of the ellipse with $x>y$.) We are going to insert a triangle into each downward step of this function, and subtract the areas of these triangles from the area of the ellipse.

We will use repeatedly (without comment) the fact that the slope of the ellipse is between $-1/2$ and $-2$ in the first quadrant, and between $-1$ and $-2$ in $\{ x>y>0 \}$.

Let $m_1$ be the largest value of $m$ for which $(m,m) \in Q$; that is, $m_1 = \lfloor R/\sqrt{3} \rfloor \geq 3$.  Hence $m_1+1 > n(m_1+1)$. Therefore $(m,n(m))$ lies in the region $\{ x>y \}$ for $m = m_1+1$, and for all larger values of $m$ too, because $n(\cdot)$ is decreasing. Further, $n(\cdot)$ is strictly decreasing for these $m$-values (until it hits zero).

Let $m_2 < R$ be the largest value of $m$ for which $n(m)$ is positive.
Notice $m_1+1 \leq m_2$, because $(m_1,m_1) \in Q$ and $m_1 \geq 3$ imply $n(m_1+1) \geq 1$. Since $n(\cdot)$ is strictly decreasing on the interval $[m_1+1,m_2]$, on each subinterval of the form $[m-1,m]$ we may sketch a triangle (shaded in Figure~\ref{figell}) with vertices at $(m,n(m)), (m-1,n(m))$ and $(m-1,n(m-1))$. These triangles lie inside the ellipse (by convexity) but outside the union of squares having upper right vertices in $Q$. The total area of the triangles is $\frac{1}{2} [n(m_1+1)-n(m_2)]$. We may include an additional triangle $T_R$ on the right end with height $n(m_2)$ and width $\frac{1}{2}n(m_2)$ (see Figure~\ref{figell}), which has area
\[
\frac{1}{4}n(m_2)^2 \geq \frac{1}{2}n(m_2) - \frac{1}{4} .
\]
The total area of the triangles including $T_R$ is $\geq \frac{1}{2} n(m_1+1) - \frac{1}{4}$.

\begin{figure}[ht]
  \begin{center}
    \begin{tikzpicture}[scale=1]
      \pgfmathsetmacro{\semi}{7.4};
      \pgfmathsetmacro{\R}{\semi*sqrt(3/2)}
      \draw (0,0) node[below] {$0$};
      \draw (\R,0) node[below] {$R$};
      \draw ({\R/sqrt(3)},{\R/sqrt(3)}) node[above] {$\left(\frac{R}{\sqrt{3}},\frac{R}{\sqrt{3}}\right)$};
      \draw (6.25,4.25) node[right] {$x^2+xy+y^2=R^2$};
      \draw (5,0) node[below] {$m_1$};
      \draw (6,0) node[below] {$m_1 \! + \! 1$};
      \draw (8,0) node[below] {$m_2$};
      \clip (-0.1,-0.1) -- (5.5,5.5) -- (10,-0.1) -- cycle;
      \foreach \x in {1,...,9}
         \foreach \y in {1,...,9}
	 {
	    \pgfmathsetmacro{\z}{\x^2+\x*\y+\y^2<3*\semi^2/2}
	    \fill[blue] (\x,\y) circle (\z*2pt);
	    \draw[blue] (\x,\y) rectangle +(-\z,-\z);
	 }
	 \begin{scope}
	 \clip (0,0) rectangle (10,10);
	 \draw[thick,red,rotate=45] (0,0) ellipse ({\semi} and {\semi*sqrt(3)});
	 \clip[rotate=45] (0,0) ellipse ({\semi} and {\semi*sqrt(3)});
      \draw (10,0) -| (0,10);
      \draw[thick] (0,0) -- (45:\semi);
	 \end{scope}
      \begin{scope}
	 \clip (0,0) rectangle (10,10);
	 \clip[shift={(-1,-1)},rotate=45] (0,0) ellipse ({\semi} and {\semi*sqrt(3)});
      \pgfmathsetmacro{\sem}{\semi-sqrt(2)}
      \end{scope}
\fill[draw=red,pattern color=red,pattern=vertical lines] (5,4) -- (5,5) -- ({\R/sqrt(3)},{\R/sqrt(3)}) -- (6,4) -- cycle;
\fill[draw=red,pattern color=red,pattern=vertical lines] (6,3) -- (6,4) -- (7,3) -- cycle;
\fill[draw=red,pattern color=red,pattern=vertical lines] (7,1) -- (7,3) -- (8,1) -- cycle;
\fill[draw=red,pattern color=red,pattern=vertical lines] (8,1) -- (8,0) -- (8.5,0) -- cycle;
      \draw (5.35,4.5) node {$T_L$};
      \draw (8.25,.3) node {$T_R$};
    \end{tikzpicture}
  \end{center}
  \caption{The upper bound on the counting function $N(\lambda)$ equals the area shown inside the ellipse minus the shaded areas, multiplied by $2$.}
  \label{figell}
\end{figure}
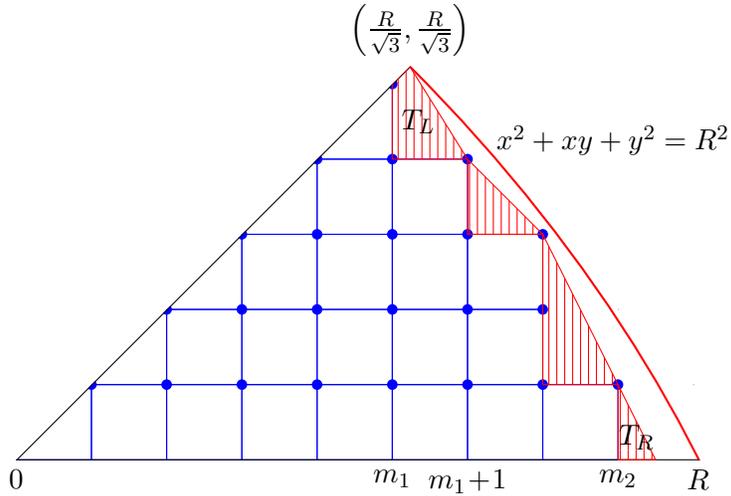

Finally, we include an additional quadrilateral $T_L$ on the left with vertices $(m_1,m_1),(m_1,n(m_1+1)),(m_1+1,n(m_1+1))$ and $(R/\sqrt{3},R/\sqrt{3})$; see Figure~\ref{figell}, which illustrates the case $n(m_1+1)=m_1-1$. (In the case $n(m_1+1)=m_1$, the quadrilateral would degenerate to a triangle.) By writing $\e = (R/\sqrt{3})-m_1 \geq 0$ and $h=m_1-n(m_1+1) \geq 0$, we see the region $T_L$ has area
\[
\e h + \frac{1}{2}\e^2 + \frac{1}{2} (1-\e)(h+\e) \geq \frac{1}{2} (h+\e) = \frac{1}{2} [(R/\sqrt{3})-n(m_1+1)].
\]
Thus the total area of the triangles and $T_L$ and $T_R$ is greater than or equal to
\begin{equation} \label{goodlower}
\frac{1}{2} \frac{R}{\sqrt{3}} - \frac{1}{4} .
\end{equation}
The same area can be found in $\{ x < y \}$, by symmetry, and so the combined area is $\geq (R/\sqrt{3}) - \frac{1}{2}$.

We subtract this area from the area of the ellipse when estimating $N(\lambda)$. Thus
\begin{equation} \label{weyltype2}
N(\lambda)< \frac{\pi R^2}{3\sqrt{3}} - \frac{R}{\sqrt{3}} + \frac{1}{2} = \frac{\sqrt{3}}{16\pi} \lambda - \frac{\sqrt{3}}{4\pi}\sqrt{\lambda} + \frac{1}{2}
\end{equation}
for all $\lambda > 48\pi^2$, as claimed in the lemma.

The lower bound on $\lambda_j(E_1)$ in the lemma now follows by solving a quadratic inequality. (Note the Weyl type estimate \eqref{weyltype} implies $\lambda_j \geq (16\pi/\sqrt{3})j$ for each $j \geq 1$. Thus when $j \geq 17$ we have $\lambda_j > 48\pi^2$, which ensures that the quadratic inequality \eqref{weyltype2} holds for all $\lambda \geq \lambda_j$.)

\medskip
Next, to obtain a lower bound on the counting function we shift the ellipse downwards by $1$ unit and leftwards by $1$ unit. The intersection of this shifted elliptical region with the first quadrant is covered by the union of squares with upper right vertices in $Q$ (since any point that lies \emph{outside} the union of squares and within the original ellipse must belong to a square whose upper right vertex is not in $Q$; notice every such square lies outside the shifted ellipse).

Therefore a lower bound on the number of lattice points in the original ellipse is given by the area of the shifted ellipse that lies in the first quadrant. To bound that area from below, we begin with the area of the original ellipse and subtract $2R-1$, which overestimates how much of that ellipse lies in the regions $0 < x < 1$ and $0 < y < 1$ (these regions will be removed from the first quadrant by the shift). A slightly better upper estimate is $2R-1-\frac{1}{2}$, because the ellipse is convex with slope $-2$ at the point $(R,0)$ and thus intersects the rectangle $[R-\frac{1}{2},R] \times [0,1]$ in an area less than $1/4$ (and argue similarly near the point $(0,R)$). Hence
\[
  N(\lambda) > \frac{\pi R^2}{3\sqrt{3}}-\left(2R-\frac{3}{2}\right) .
\]

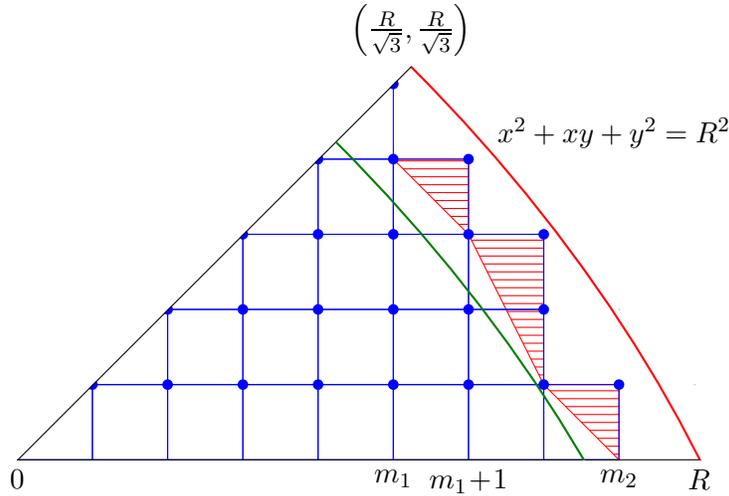
\begin{figure}[ht]
  \begin{center}
    \begin{tikzpicture}[scale=1]
      \pgfmathsetmacro{\semi}{7.4};
      \pgfmathsetmacro{\R}{\semi*sqrt(3/2)}
      \draw (0,0) node[below] {$0$};
      \draw (\R,0) node[below] {$R$};
      \draw ({\R/sqrt(3)},{\R/sqrt(3)}) node[above] {$\left(\frac{R}{\sqrt{3}},\frac{R}{\sqrt{3}}\right)$};
      \draw (6.25,4.35) node[right] {$x^2+xy+y^2=R^2$};
      \draw (5,0) node[below] {$m_1$};
      \draw (6,0) node[below] {$m_1 \! + \! 1$};
      \draw (8,0) node[below] {$m_2$};
\fill[draw=red,pattern color=red,pattern=horizontal lines] (5,4) -- (6,4) -- (6,3) -- cycle;
\fill[draw=red,pattern color=red,pattern=horizontal lines] (6,3) -- (7,3) -- (7,1) -- cycle;
\fill[draw=red,pattern color=red,pattern=horizontal lines] (7,1) -- (8,1) -- (8,0) -- cycle;
      \clip (-0.1,-0.1) -- (5.5,5.5) -- (10,-0.1) -- cycle;
      \foreach \x in {1,...,9}
         \foreach \y in {1,...,9}
	 {
	    \pgfmathsetmacro{\z}{\x^2+\x*\y+\y^2<3*\semi^2/2}
	    \fill[blue] (\x,\y) circle (\z*2pt);
	    \draw[blue] (\x,\y) rectangle +(-\z,-\z);
	 }
	 \begin{scope}
	 \clip (0,0) rectangle (10,10);
	 \draw[thick,red,rotate=45] (0,0) ellipse ({\semi} and {\semi*sqrt(3)});
	 \clip[rotate=45] (0,0) ellipse ({\semi} and {\semi*sqrt(3)});
      \draw (10,0) -| (0,10);
      \draw[thick] (0,0) -- (45:\semi);
	 \end{scope}
      \begin{scope}
	 \clip (0,0) rectangle (10,10);
	 \draw[green!50!black,thick,shift={(-1,-1)},rotate=45] (0,0) ellipse ({\semi} and {\semi*sqrt(3)});
	 \clip[shift={(-1,-1)},rotate=45] (0,0) ellipse ({\semi} and {\semi*sqrt(3)});
      \pgfmathsetmacro{\sem}{\semi-sqrt(2)}
      \end{scope}
    \end{tikzpicture}
  \end{center}
  \caption{The lower bound on the counting function $N(\lambda)$ equals the area inside the shifted ellipse plus the shaded areas, multiplied by $2$.}
  \label{figlower}
\end{figure}

An even better lower bound can be obtained by adding to the right side of the inequality the areas of the triangles shaded in Figure~\ref{figlower}. These triangles are inside the union of squares but outside the shifted ellipse. To specify these triangles precisely, let $m$ be between $m_1+1$ and $m_2$ (inclusive), and consider the triangle with vertices $(m,n(m)),(m-1,n(m))$ and $(m,n(m+1))$. The shifted ellipse passes below or through each vertex of the triangle. We will show that the triangle lies above the shifted ellipse, by considering two cases. (i) If $n(\cdot)$ decreases by $1$ at $m$, meaning $n(m)=n(m+1)+1$, then the hypotenuse of the triangle has slope $-1$, whereas the ellipse has slope less than $-1$ and passes below the upper left vertex at $(m-1,n(m))$; hence the triangle lies above the shifted ellipse. (ii) If $n(\cdot)$ decreases by $2$ at $m$, meaning $n(m)=n(m+1)+2$, then the hypotenuse of the triangle has slope $-2$, whereas the ellipse has slope greater than $-2$ and passes below the lower right vertex at $(m,n(m+1))$; hence the triangle lies above the shifted ellipse.

The triangles have total area $\frac{1}{2}n(m_1+1)$. Recall that $n(m_1+1)$ equals $m_1$ or $m_1-1$ or $m_1-2$. In the first case, the triangles have area $\frac{1}{2}m_1$. In the second they have area $\frac{1}{2}(m_1-1)$. In the third case, where $n(m_1+1)=m_1-2$, we employ also a triangle (not shown in Figure~\ref{figlower}) with vertices at $(m_1,m_1),(m_1,m_1-2)$ and $(m_1-\frac{1}{2},m_1-1)$; this triangle has area $\frac{1}{2}$, and so the areas sum to $\frac{1}{2}(m_1-2)+\frac{1}{2}$. Hence in every case, we obtain an area of at least $\frac{1}{2}(m_1-1)$. Doubling this estimate due to symmetry, and then estimating $m_1$ from below with $(R/\sqrt{3})-1$, we obtain the lower estimate
\[
  N(\lambda) > \frac{\pi R^2}{3\sqrt{3}}-\left(2R-\frac{3}{2}\right) + \frac{R}{\sqrt{3}} - 2 ,
\]
which proves the lower bound on $N(\lambda)$ in the lemma. The upper bound on $\lambda_j(E_1)$ then follows easily.
\end{proof}

Lattice counting problems are well studied in number theory; see the monograph by Huxley \cite{H96}. Analytic methods using exponential sums give the best asymptotic results, but more elementary geometric methods like in the preceding proof have the advantage of yielding explicit inequalities for all $\lambda$.

\medskip
Let $\lambda_j^a(E_1)$ denote the $j$-th antisymmetric eigenvalue of the equilateral triangle, and write $N^a(\lambda)$ for the corresponding counting function.

\begin{lemma}\label{antisym}
The antisymmetric counting function satisfies
\[
N^a(\lambda)
< \frac{\sqrt{3}}{32\pi} \lambda - \frac{\sqrt{3}}{4\pi} \sqrt{\lambda} + \frac{3}{4} , \qquad \forall \, \lambda > 48\pi^2 .
\]
Hence for all $j \geq 9$,
\begin{align*}
\lambda_j^a(E_1)
& \geq \frac{32\pi}{\sqrt{3}} \Big( j-\frac{3}{4} \Big)+ 8 \sqrt{\frac{32\pi}{\sqrt{3}} \Big( j-\frac{3}{4} \Big) + 16} + 32 \\
& > 58j+8\sqrt{58j-28}-12 .
\end{align*}
\end{lemma}
\begin{proof}[Proof of Lemma \ref{antisym}]
The antisymmetric eigenvalues of the equilateral triangle $E_1$ are
\[
  \sigma_{m,n}^a = \frac{16\pi^2}{9}(m^2+mn+n^2), \qquad m>n \geq 1 .
\]
Thus in the notation of the previous proof, the task is to count lattice points in the first quadrant that lie inside the ellipse $x^2+xy+y^2 = R^2$ and in the region $\{ x>y \}$. The area of this region is $\pi R^2/6\sqrt{3}$, which gives an initial estimate $N^a(\lambda) < \pi R^2/6\sqrt{3} = (\sqrt{3}/32\pi)\lambda$.

To improve the bound, we subtract an area $\frac{1}{2} (R/\sqrt{3})-\frac{1}{4}$ like in \eqref{goodlower} of the previous proof. We may also subtract the triangles adjacent to the diagonal $\{ x=y \}$; there are $m_1$ such triangles, each having area $\frac{1}{2}$. Thus the total area to be subtracted is $\frac{1}{2} (R/\sqrt{3}) - \frac{1}{4} + \frac{1}{2} \lfloor R/\sqrt{3} \rfloor$, which is greater than $(R/\sqrt{3})-\frac{3}{4}$. Therefore
\[
N^a(\lambda)
< \frac{\pi R^2}{6\sqrt{3}} - \frac{R}{\sqrt{3}} + \frac{3}{4} \\
= \frac{\sqrt{3}}{32\pi} \lambda - \frac{\sqrt{3}}{4\pi} \sqrt{\lambda} + \frac{3}{4} .
\]
Inverting this estimate on the counting function yields a bound on $\lambda_j^a$. This bound holds for $j \geq 9$, because the initial estimate $N^a(\lambda) < (\sqrt{3}/32\pi)\lambda$ ensures that $\lambda^a_j > 48\pi^2$ when $j \geq 9$, as required for the method above.
\end{proof}

\begin{proof}[Proof of Lemma \ref{compequilateral}]
The right triangle $F=T(1,2\sqrt{3})$ is half of an equilateral triangle, and so its Dirichlet spectrum is the antisymmetric spectrum of that equilateral triangle. Also, $E=T(0,\sqrt{3})$ is an equilateral triangle. Therefore to prove $\left. \Lambda_n D^2 \right|_F > \frac{11}{6} \! \left. \Lambda_n D^2 \right|_E$, it suffices (by scale invariance) to show
\[
    \Lambda_n^a > \frac{11}{6} \Lambda_n
\]
for the equilateral triangle of sidelength $1$.

For $n \leq 110$ we verify the inequality directly, in Lemma~\ref{explicit}. For larger values of $n$, we can estimate the ratio of eigenvalues by Lemmas~\ref{counting} and \ref{antisym}:
\[
    \frac{\lambda_n^a}{\lambda_n} > \frac{58n+8\sqrt{58n-28}-12}{(29.03)n+9.9\sqrt{(29.03)n+39}+64} ,
\]
which is larger than $\frac{11}{6}$ when $n \geq 110$, by elementary estimates. Hence $\Lambda_n^a > \frac{11}{6} \Lambda_n$ for all $n$.
\end{proof}

\section{\bf Proof of Theorem~\ref{th:secondtone}: the lower bound on $\lambda_2 D^2$} \label{sec:secondtone_proof}

Equality does hold for equilaterals, by explicit evaluation of the eigenvalues as in Appendix~\ref{equilateral}. Consider a non-equilateral triangle. We may assume it is subequilateral by domain monotonicity (see Step 1 of Section~\ref{sectraces}).

The desired inequality for $\lambda_1 D^2$ was proved already in Theorem~\ref{th:trace}. Thus our task is to prove
\[
\lambda_2 D^2 > \frac{7 \cdot 16 \pi^2}{9}
\]
for all subequilateral triangles. A numerical verification can be found in Figure~\ref{fig:tones}b later in the paper; in this section we develop a rigorous proof.

In terms of the subequilateral triangle $T(0,b)$ with diameter $\sqrt{b^2+1}$, we need to show
\begin{equation} \label{subeqneeded}
\lambda_2(0,b) > \frac{112\pi^2}{9(b^2+1)} , \qquad b>\sqrt{3} .
\end{equation}

First assume $b \geq 5/2$. The subequilateral triangle $T(0,b)$ has aperture angle $\beta=2\arctan(1/b)$, and hence lies in a sector of radius $D=\sqrt{b^2+1}$ and aperture $\alpha=2\arctan \big( 1/(5/2) \big) \simeq 0.761$. The second eigenvalue of the sector is $j_{\nu,2}^2/D^2$ where $\nu=\pi/\alpha$, with eigenfunction $J_\nu(j_{\nu,2}r/D) \cos (\nu \theta)$.
Hence by domain monotonicity,
\[
\lambda_2(0,b) \geq \frac{j_{\nu,2}^2}{D^2} \simeq \frac{126}{D^2} .
\]
The value $126$ exceeds $112\pi^2/9 \simeq 123$, and so \eqref{subeqneeded} is proved.

\medskip
Now assume $b \leq 5/2$. The idea is to sharpen our Method of the Unknown Trial Function by using an ``endpoint'' domain that is subequilateral, rather than the right triangle used in proving Theorem~\ref{th:trace}. The price to be paid is that the eigenvalues of the subequilateral triangle must be estimated somehow (in Lemma~\ref{evalCh} below), whereas the eigenvalues of the right triangle were known explicitly.

We start with an upper bound for the fundamental tone of a triangle in terms of the side lengths:
\begin{equation} \label{polyaup}
\lambda_1 \leq \frac{\pi^2}{3} \, \frac{l_1^2 + l_2^2 + l_3^2}{A^2} .
\end{equation}
This bound is due to P\'{o}lya \cite{P52}; see the discussion in \cite[{\S}3]{LS10b}. Applying the bound to $T(0,b)$ yields
\[
\lambda_1(0,b) \leq \frac{2\pi^2(b^2+3)}{3b^2} .
\]
We will combine this upper bound on $\lambda_1$ with a lower bound on $\lambda_1 + \lambda_2$. First, decompose
\begin{align*}
\lambda_2 = (\lambda_1 + \lambda_2) - \lambda_1 & = \Lambda_2 - \lambda_1 \\
& \geq \Lambda_2 - \frac{2\pi^2(b^2+3)}{3b^2} .
\end{align*}
Then observe that \eqref{subeqneeded} will follow if we prove
\[
\Lambda_2(0,b) > \frac{112\pi^2}{9(b^2+1)} + \frac{2\pi^2(b^2+3)}{3b^2} , \qquad b>\sqrt{3} .
\]
We rewrite this desired inequality as
\begin{equation} \label{whatwewant}
\Lambda_2(0,b) > C(b) \Lambda_2(0,\sqrt{3}) , \qquad b>\sqrt{3} ,
\end{equation}
where
\[
C(b) = \frac{3b^4 + 68b^2 + 9}{20b^2(b^2+1)}
\]
and where we have used that $\Lambda_2(0,\sqrt{3})=40\pi^2/9$ by Appendix~\ref{equilateral} (since $T(0,\sqrt{3})$ is equilateral with diameter $2$). As a check, notice that equality holds in \eqref{whatwewant} for the equilateral case, $b=\sqrt{3}$, because $C(\sqrt{3})=1$.

We finish the proof with the following two lemmas. They use the quantity
\[
\widetilde{C}(b) = \frac{13b^2+81}{40b^2} ,
\]
which is greater than or equal to $C(b)$ (although we will not need that fact).
\begin{lemma} \label{condCh}
Fix $h > \sqrt{3}$. If
\begin{equation} \label{lemmahyp}
\Lambda_2(0,h) \geq \widetilde{C}(h) \Lambda_2(0,\sqrt{3})
\end{equation}
then
\[
\Lambda_2(0,b) > C(b) \Lambda_2(0,\sqrt{3})
\]
whenever $\sqrt{3} < b < h$.
\end{lemma}
\begin{lemma} \label{evalCh}
Let $h = 5/2$. Then $\Lambda_2(0,h) > \widetilde{C}(h) \Lambda_2(0,\sqrt{3})$.
\end{lemma}
\begin{proof}[Proof of Lemma~\ref{condCh}] Suppose $\sqrt{3} < b < h$.
By Lemma \ref{lemtrace} with $n=2$, we have $\Lambda_2(0,b) > C(b) \Lambda_2(0,\sqrt{3})$ if
\[
(1-\gamma_2) + \frac{b^2}{3} \gamma_2 < \frac{1}{C(b)} ,
\]
which is equivalent to
\begin{equation} \label{Ceh1}
  \gamma_2 < 3 \Big( \frac{1}{C(b)} - 1 \Big) \Big/ (b^2-3) .
\end{equation}
On the other hand, by Lemma \ref{lemtrace} and hypothesis \eqref{lemmahyp}, we have
\[
\Lambda_2(0,b) > \frac{C(b)}{\widetilde{C}(h)} \Lambda_2(0,h) \geq C(b) \Lambda_2(0,\sqrt{3})
\]
if
\[
(1-\gamma_2) + \frac{b^2}{h^2} \gamma_2 < \frac{\widetilde{C}(h)}{C(b)} ,
\]
which is equivalent to
\begin{equation} \label{Ceh2}
  \gamma_2 > h^2 \Big( 1 - \frac{\widetilde{C}(h)}{C(b)} \Big) \Big/ (h^2-b^2) .
\end{equation}
Our task is to prove that at least one of \eqref{Ceh1} or \eqref{Ceh2} holds, which is guaranteed if the right side of \eqref{Ceh1} is greater than the right side of \eqref{Ceh2}. Thus we want to prove
\[
  3 \Big( \frac{1}{C(b)} - 1 \Big) \Big/ (b^2-3) > h^2 \Big( 1 - \frac{\widetilde{C}(h)}{C(b)} \Big) \Big/ (h^2-b^2) ,
\]
which simplifies to
\[
\widetilde{C}(h) > \frac{3(h^2 + 17)}{20 h^2} + \frac{7 (h^2 - 3)}{10 h^2 (b^2 + 1)} .
\]
The right side is a strictly decreasing function of $b$, and so it suffices to verify that the inequality holds with ``$\geq$'' at $b=\sqrt{3}$. Indeed it holds there with equality, by definition of $\widetilde{C}(h)$.
\end{proof}

\begin{proof}[Proof of Lemma~\ref{evalCh}]
P\'{o}lya and Szeg\H{o}'s result of Faber--Krahn type says that $\lambda_1 A \geq 4\pi^2/\sqrt{3}$ for all triangles, with equality for the equilateral triangle \cite[Note A]{PS51}. Applying this bound to the subequilateral $T(0,h)$ gives that
\[
\lambda_1(0,h) \geq \frac{4\pi^2}{\sqrt{3}h} .
\]
Hence our goal $\Lambda_2(0,h) > \widetilde{C}(h) \Lambda_2(0,\sqrt{3})$ will follow if we establish the lower bound
\begin{equation} \label{nineteen35}
  \lambda_2(0,h) \geq \frac{13h^2+81}{40h^2} \, \frac{40\pi^2}{9} - \frac{4\pi^2}{\sqrt{3}h} \simeq 19.35 ,
\end{equation}
where the final line used that $\Lambda_2(0,\sqrt{3}) = 40\pi^2/9$ and $h=5/2$.

Our task is to rigorously estimate the second eigenvalue of the single triangle $T(0,h)$. We will employ a method discovered by Fox, Henrici and Moler \cite{FHM}, or rather, a subsequent improvement due to Moler and Payne \cite[Theorem~3]{PM}. Consider a bounded plane domain $\Omega$ with piecewise smooth boundary and area $A$. Let $\bar{u}$ be an eigenfunction of the Laplacian with eigenvalue $\bar{\lambda}$ that satisfies the Dirichlet boundary condition only approximately. More precisely, assume $-\Delta \bar{u} = \bar{\lambda} \bar{u}$ on $\Omega$, and that $\bar{u}$ is smooth on $\overline{\Omega}$. Then there exists a Dirichlet eigenvalue $\lambda$ on $\Omega$ such that
\begin{equation} \label{fox}
\frac{\bar{\lambda}}{1-\epsilon} \geq \lambda \geq \frac{\bar{\lambda}}{1+\epsilon}
\end{equation}
where
\[
    \epsilon=\frac{\sqrt{A} \lVert \bar{u} \rVert_{L^\infty(\partial \Omega)}}{\lVert \bar{u} \rVert_{L^2(\Omega)}}.
\]

Let $\Omega$ be a copy of the triangle $T(0,h)$ that has been translated to move the vertex $(0,h)$ to the origin, and then rotated to make the triangle symmetric with respect to the positive $x$-axis. The area equals $h$. We need to find a suitable eigenfunction $\bar{u}$ that approximately satisfies Dirichlet boundary conditions on $\partial \Omega$. A method for constructing this function is described by Fox \emph{et al.}\ \cite{FHM}: one should take an optimized linear combination of eigenfunctions of a circular sector. Here we find a good choice to be
\begin{align*}
\bar{u}(r,\theta)
& = J_\nu \! \left(\frac{334r}{75}\right)\cos(\nu \theta) + \frac{5}{22}J_{3\nu} \! \left(\frac{334r}{75}\right) \cos(3\nu \theta) \\
& \qquad - \frac{2225}{53} J_{5\nu} \! \left(\frac{334r}{75}\right) \cos(5\nu \theta),
\end{align*}
where $\nu=\pi/\alpha$ and $\alpha=2\arctan(1/h) \simeq 0.761$ is the aperture of the triangle. The eigenvalue for $\bar{u}$ is $\bar{\lambda}=(334/75)^2 \simeq 19.832$. (The motivation for our choice of $\bar{u}$ is as follows. The second eigenfunction of the sector of radius $h$ and aperture $\alpha$ is $J_\nu(j_{\nu,2}r/h)\cos(\nu \theta)$, with eigenvalue $(j_{\nu,2}/h)^2$. One computes that $j_{\nu,2}/h \simeq 4.49$, which is close to $334/75 \simeq 4.45$. The other two functions in the linear combination for $\bar{u}$ are higher symmetric eigenfunctions of a sector; their role is to reduce the magnitude of $\bar{u}$ on the short side of the isosceles triangle, in other words, to compensate for the fact that we need an approximate Dirichlet condition on a triangle and not on a sector. Lastly, the rational value $334/75$ is chosen to come close to minimizing the error $\epsilon$ below.)

Now we estimate $\epsilon$. First calculate the $L^2$-norm of $\bar{u}$ on the circular sector with radius $h$ and aperture $\alpha$; the norm is greater than $0.25$. This sector lies inside $\Omega$, and so
\[
  \lVert\bar{u}\rVert_{L^2(\Omega)} > 0.25 .
\]
Next, the function $\bar{u}$ equals $0$ on the two equal sides of the isosceles triangle $\Omega$, and the third side has polar equation $r=h/\cos\theta$. Thus the $L^\infty$ norm of $\bar{u}$ on $\partial \Omega$ is
\[
\lVert \bar{u} \rVert_{L^\infty(\partial \Omega)} = \max_{|\theta| \leq \alpha/2} |\bar{u}(h/\cos\theta,\theta)| < 0.0013 ,
\]
by a numerical estimate. (Instead of finding this maximum numerically, one could find a rigorous upper bound by dividing the boundary into small intervals and then evaluating $\bar{u}$ at one point of each interval and employing crude global estimates on the derivatives of $\bar{u}$. Thus the estimation of the $L^\infty$ norm can, in principle, be achieved using only finitely many function evaluations.)

Combining the last two estimates gives that $\epsilon < 0.009$, and so \eqref{fox} guarantees the existence of an eigenvalue $\lambda$ of $\Omega$ with
\[
  20.03 > \frac{19.84}{1-0.009} > \lambda > \frac{19.83}{1+0.009} > 19.65 .
\]
This lower bound is well above the value $19.35$ in \eqref{nineteen35}. Thus if $\lambda$ is the first or second eigenvalue of $\Omega$, then we are done. Fortunately, $\lambda$ cannot be the third (or higher) eigenvalue, by domain monotonicity, since the third Dirichlet eigenvalue of the circular sector containing $T(0,h)$ is $j_{2\nu,1}^2/(1+h^2) \simeq 21.6$, which is larger than $\lambda$.
\end{proof}


\section{\bf Monotonicity of eigenvalues for isosceles triangles}
\label{sec:isos}

In this section we present monotonicity results for low eigenvalues of isosceles triangles.

Consider the isosceles triangle $T(\alpha)$ having aperture
$\alpha \in (0,\pi)$ and equal sides of length $l$, with vertex at the origin.
After rotating the triangle to make it symmetric about the
positive $x$-axis, it can be written as
\[
T(\alpha) = \big\{ (x,y) : 0<x<l\cos(\alpha/2), |y|<x\tan(\alpha/2) \big\} .
\]

Write $\lambda_1(\alpha), \lambda_a(\alpha)$ and $\lambda_s(\alpha)$ for the fundamental tone of $T(\alpha)$, the smallest antisymmetric eigenvalue, and the smallest symmetric eigenvalue greater than $\lambda_1$, respectively. These eigenvalues are plotted numerically in Figure~\ref{figl}, normalized by the square of the sidelength. Figure~\ref{figd} plots them again, this time normalized by the square of the diameter; notice the corners appearing at $\alpha=\pi/3$, due to the diameter switching from $l$ (the length of the two equal sides) to $2l\sin(\alpha/2)$ (the length of the third side) as the aperture passes through $\pi/3$.

\begin{figure}[t]
  \begin{center}
      \subfloat[Sidelength scaling: $\lambda l^2$]{
    \begin{tikzpicture}[scale=6,smooth]
      \draw[<->] (pi/6,1) |- (2.15,0);
      \clip (0.3,1) rectangle (2.2,-0.1);
      \draw (pi/6,0) node[below] {$\frac{\pi}{6}$};
      \draw (pi/3,0) node[below] {$\frac{\pi}{3}$};
      \draw (pi/2,0) node[below] {$\frac{\pi}{2}$};
      \draw (2*pi/3,0) node[below] {$\frac{2\pi}{3}$};
      \draw (2.18,0) node[below] {$\alpha$};
      \draw (pi/6,0) node[left] {$0$};
      \draw[dotted] (pi/3,0) -- (pi/3,1);
      \draw[dotted] (pi/2,0) -- (pi/2,1);
      \draw [mark=+, mark size=0.2] plot coordinates{(pi/6,0) (pi/3,0) (pi/2,0) (2*pi/3,0)};
      \draw [yscale=0.005] plot coordinates{
( 0.5236,104.9618)( 0.5760,93.2249)( 0.6283,84.0877)( 0.6807,76.8433)( 0.7330,71.0173)( 0.7854,66.2811)( 0.8378,62.4011)( 0.8901,59.2077)( 0.9425,56.5755)( 0.9948,54.4103)( 1.0472,52.6405)( 1.0996,51.2118)( 1.1519,50.0822)( 1.2043,49.2199)( 1.2566,48.6008)( 1.3090,48.2074)( 1.3614,48.0274)( 1.4137,48.0531)( 1.4661,48.2813)( 1.5184,48.7125)( 1.5708,49.3512)( 1.6232,50.2058)( 1.6755,51.2890)( 1.7279,52.6181)( 1.7802,54.2160)( 1.8326,56.1117)( 1.8850,58.3418)( 1.9373,60.9524)( 1.9897,64.0011)( 2.0420,67.5602)( 2.0944,71.7212)
     };
     \fill ( 1.3614,48.0274*0.005) circle (0.25pt) node [below] {\tiny $48.03$};
     \draw[dotted] ( pi/6,122.8361*0.005) node[left] {$\frac{7\cdot16\pi^2}{9}$}-- ++(pi/6,0);
      \draw [yscale=0.005] plot coordinates{
( 0.5236,196.3239)( 0.5760,179.5057)( 0.6283,166.3514)( 0.6807,155.9147)( 0.7330,147.5537)( 0.7854,140.8197)( 0.8378,135.3910)( 0.8901,131.0326)( 0.9425,127.5697)( 0.9948,124.8708)( 1.0472,122.8361) ( 1.0996,121.3892)( 1.1519,120.4715)( 1.2043,120.0379)( 1.2566,120.0536)( 1.3090,120.4915)( 1.3614,121.3300)( 1.4137,122.5504)( 1.4661,124.1358)( 1.5184,126.0681)( 1.5708,128.3272)( 1.6232,130.8909)( 1.6755,133.7371)( 1.7279,136.8506)( 1.7802,140.2328)( 1.8326,143.9114)( 1.8850,147.9463)( 1.9373,152.4287)( 1.9897,157.4771)( 2.0420,163.2337)( 2.0944,169.8653)
      };
     \fill ( 1.2243,120.0379*0.005) circle (0.25pt)node [below] {\tiny $120.04$};
      \draw [yscale=0.005,dashed] plot coordinates{
( 0.5236,293.5534)( 0.5760,255.3847)( 0.6283,225.8126)( 0.6807,202.4296)( 0.7330,183.6308)( 0.7854,168.3112)( 0.8378,155.6893)( 0.8901,145.1994)( 0.9425,136.4238)( 0.9948,129.0490)( 1.0472,122.8361) ( 1.0996,117.6014)( 1.1519,113.2018)( 1.2043,109.5256)( 1.2566,106.4854)( 1.3090,104.0127)( 1.3614,102.0543)( 1.4137,100.5700)( 1.4661,99.5300)( 1.5184,98.9142)( 1.5708,98.7107)( 1.6232,98.9155)( 1.6755,99.5328)( 1.7279,100.5742)( 1.7802,102.0601)( 1.8326,104.0202)( 1.8850,106.4950)( 1.9373,109.5376)( 1.9897,113.2165)( 2.0420,117.6195)( 2.0944,122.8585)
      };
     \fill ( 1.5708,98.7107*0.005) circle (0.25pt);
     \draw[dotted] ( pi/6,98.7107*0.005) node[left] {$10\pi^2$}-- ++(pi/3,0);
     \draw[dotted] ( pi/6,52.6405*0.005) node[left] {$\frac{3\cdot16\pi^2}{9}$}-- ++(pi/6,0);
    \end{tikzpicture}
  \label{figl}
    }

    \subfloat[Diameter scaling: $\lambda D^2$]{
    \begin{tikzpicture}[scale=6]
      \draw[<->] (pi/6,1) |- (2.15,0);
      \clip (0.3,1) rectangle (2.2,-0.1);
      \draw (pi/6,0) node[below] {$\frac{\pi}{6}$};
      \draw (pi/3,0) node[below] {$\frac{\pi}{3}$};
      \draw (pi/2,0) node[below] {$\frac{\pi}{2}$};
      \draw (2*pi/3,0) node[below] {$\frac{2\pi}{3}$};
      \draw (2.18,0) node[below] {$\alpha$};
      \draw (pi/6,0) node[left] {\small $0$};
      \draw[dotted] (pi/3,0) -- (pi/3,1);
      \draw[dotted] (pi/2,0) -- (pi/2,1);
      \draw [mark=+, mark size=0.2] plot coordinates{(pi/6,0) (pi/3,0) (pi/2,0) (2*pi/3,0)};
      \draw [yscale=0.003,smooth] plot coordinates{
      ( 0.5236,196.3239)( 0.5760,179.5057)( 0.6283,166.3514)( 0.6807,155.9147)( 0.7330,147.5537)( 0.7854,140.8197)( 0.8378,135.3910)( 0.8901,131.0326)( 0.9425,127.5697)( 0.9948,124.8708)( 1.0472,122.8361)
      };
      \draw [yscale=0.003,smooth] plot coordinates{
      ( 1.0472,122.8361)( 1.0996,132.5593)( 1.1519,142.9427)( 1.2043,154.0403)( 1.2566,165.9100)( 1.3090,178.6121)( 1.3614,192.2081)( 1.4137,206.7587)( 1.4661,222.3201)( 1.5184,238.9403)( 1.5708,256.6544)( 1.6232,275.4823)( 1.6755,295.4329)( 1.7279,316.5176)( 1.7802,338.7777)( 1.8326,362.3168)( 1.8850,387.3284)( 1.9373,414.1086)( 1.9897,443.0577)( 2.0420,474.6806)( 2.0944,509.5960)
     };
\fill ( 1.0472,122.8361*0.003) circle (0.25pt);
\draw[dotted] ( pi/6,122.8361*0.003) node[left] {$\frac{7\cdot 16\pi^2}{9}$}-- ++(pi/6,0);
      \draw [yscale=0.003, dashed,smooth] plot coordinates{
( 0.5236,293.5534)( 0.5760,255.3847)( 0.6283,225.8126)( 0.6807,202.4296)( 0.7330,183.6308)( 0.7854,168.3112)( 0.8378,155.6893)( 0.8901,145.1994)( 0.9425,136.4238)( 0.9948,129.0490)( 1.0472,122.8361)
};
      \draw [yscale=0.003, dashed,smooth] plot coordinates{
( 1.0472,122.8361)( 1.0996,128.4229)( 1.1519,134.3169)( 1.2043,140.5503)( 1.2566,147.1592)( 1.3090,154.1844)( 1.3614,161.6721)( 1.4137,169.6748)( 1.4661,178.2526)( 1.5184,187.4749)( 1.5708,197.4213)( 1.6232,208.1848)( 1.6755,219.8736)( 1.7279,232.6149)( 1.7802,246.5592)( 1.8326,261.8853)( 1.8850,278.8076)( 1.9373,297.5847)( 1.9897,318.5317)( 2.0420,342.0353)( 2.0944,368.5754)
      };
      \draw [yscale=0.003,smooth] plot coordinates{
      ( 0.5236,104.9618)( 0.5760,93.2249)( 0.6283,84.0877)( 0.6807,76.8433)( 0.7330,71.0173)( 0.7854,66.2811)( 0.8378,62.4011)( 0.8901,59.2077)( 0.9425,56.5755)( 0.9948,54.4103)( 1.0472,52.6405)
      };
       \draw [yscale=0.003,smooth] plot coordinates{
      ( 1.0472,52.6405)( 1.0996,55.9242)( 1.1519,59.4238)( 1.2043,63.1621)( 1.2566,67.1647)( 1.3090,71.4608)( 1.3614,76.0838)( 1.4137,81.0719)( 1.4661,86.4691)( 1.5184,92.3262)( 1.5708,98.7024)( 1.6232,105.6667)( 1.6755,113.3003)( 1.7279,121.6988)( 1.7802,130.9763)( 1.8326,141.2689)( 1.8850,152.7408)( 1.9373,165.5916)( 1.9897,180.0653)( 2.0420,196.4637)( 2.0944,215.1635)
      };
\fill ( 1.0472,52.6405*0.003) circle (0.25pt);
\draw[dotted] ( pi/6,52.6405*0.003) node[left] {$\frac{3\cdot 16\pi^2}{9}$}-- ++(pi/6,0);
    \end{tikzpicture}
  \label{figd}
    }
  \end{center}
  \caption{Plots of the fundamental tone, the smallest antisymmetric tone (dashed), and the smallest symmetric tone larger than the fundamental tone, for the isosceles triangle $T(\alpha)$ with aperture $\alpha$. Global minimum points are indicated with dots.}
  \label{fig:tones}
\end{figure}
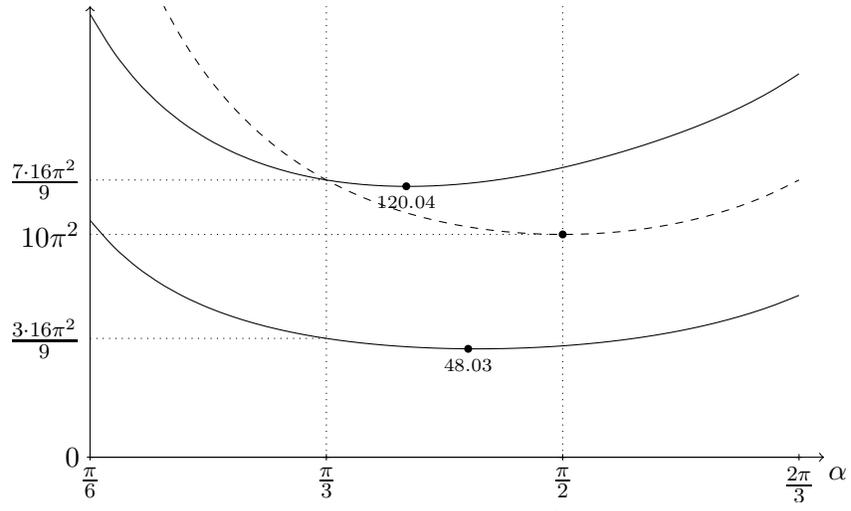
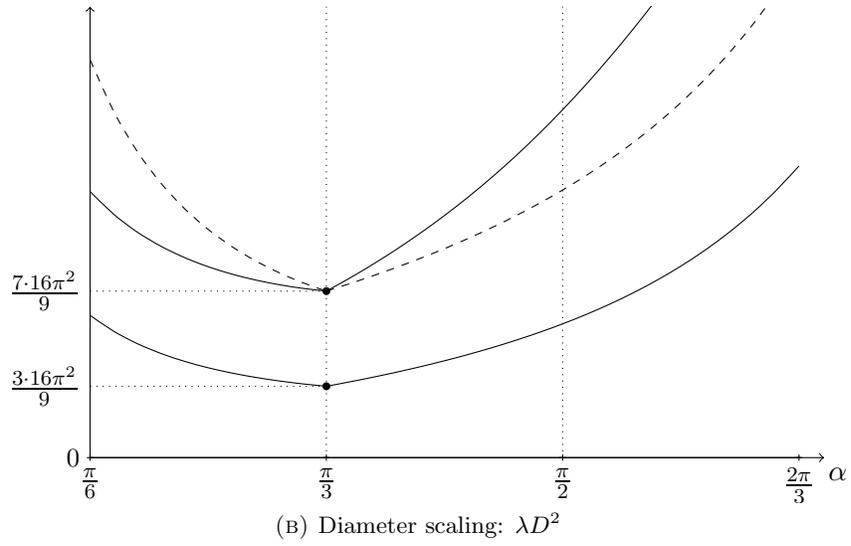

\afterpage{\clearpage}

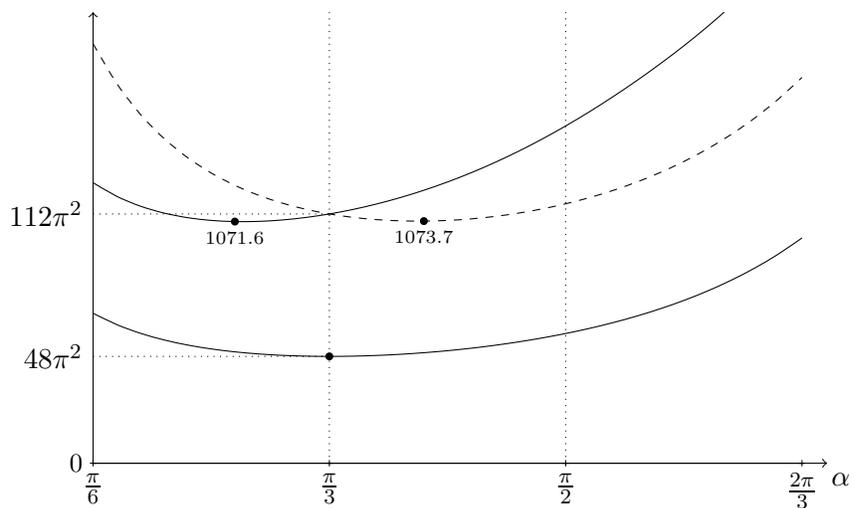
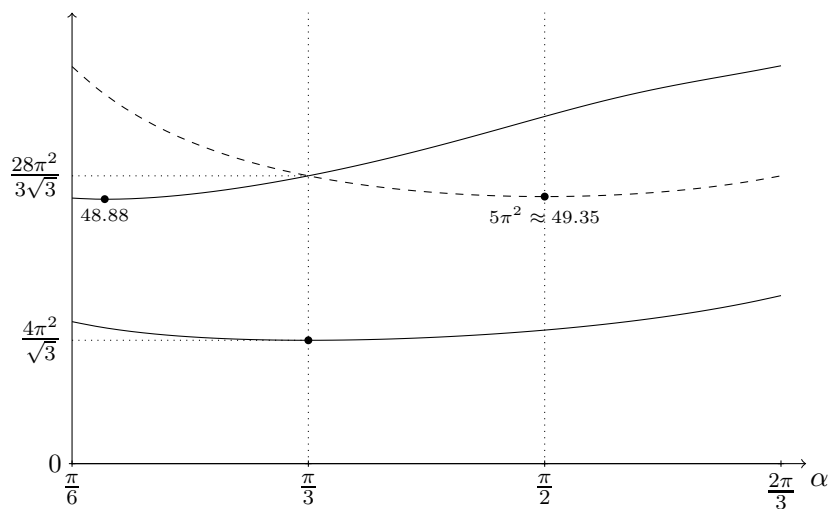
\begin{figure}[t]
  \begin{center}
    \subfloat[Perimeter scaling: $\lambda L^2$]{
    \begin{tikzpicture}[scale=6,smooth]
      \draw[<->] (pi/6,1) |- (2.15,0);
      \clip (0.3,1) rectangle (2.2,-0.1);
      \draw (pi/6,0) node[below] {$\frac{\pi}{6}$};
      \draw (pi/3,0) node[below] {$\frac{\pi}{3}$};
      \draw (pi/2,0) node[below] {$\frac{\pi}{2}$};
      \draw (2*pi/3,0) node[below] {$\frac{2\pi}{3}$};
      \draw (2.18,0) node[below] {$\alpha$};
      \draw (pi/6,0) node[left] {\small $0$};
      \draw[dotted] (pi/3,0) -- (pi/3,1);
      \draw[dotted] (pi/2,0) -- (pi/2,1);
      \draw [mark=+, mark size=0.2] plot coordinates{(pi/6,0) (pi/3,0) (pi/2,0) (2*pi/3,0)};
      \draw [yscale=0.05,smooth] plot coordinates{
( 0.5236,12.4425)( 0.5760,11.8367)( 0.6283,11.4007)( 0.6807,11.0940)( 0.7330,10.8894)( 0.7854,10.7678)( 0.8378,10.7161)( 0.8901,10.7246)( 0.9425,10.7868)( 0.9948,10.8978)
( 1.0472,11.0543)( 1.0996,11.2542)( 1.1519,11.4963)( 1.2043,11.7801)( 1.2566,12.1054)( 1.3090,12.4726)( 1.3614,12.8824)( 1.4137,13.3354)( 1.4661,13.8321)( 1.5184,14.3728)( 1.5708,14.9570)( 1.6232,15.5838)( 1.6755,16.2521)( 1.7279,16.9611)( 1.7802,17.7120)( 1.8326,18.5091)( 1.8850,19.3613)( 1.9373,20.2819)( 1.9897,21.2883)( 2.0420,22.4022)( 2.0944,23.6493)
     };
     \fill ( 0.8378,10.7161*0.05) circle (0.25pt) node [below] {\tiny $1071.6$};
\draw[dotted] ( pi/6,11.0543*0.05) node[left] {$112\pi^2$}-- ++(pi/6,0);
      \draw [yscale=0.05, dashed,smooth] plot coordinates{
( 0.5236,18.6046)( 0.5760,16.8402)( 0.6283,15.4758)( 0.6807,14.4038)( 0.7330,13.5518)( 0.7854,12.8700)( 0.8378,12.3227)( 0.8901,11.8842)( 0.9425,11.5355)( 0.9948,11.2624)
( 1.0472,11.0543)( 1.0996,10.9030)( 1.1519,10.8026)( 1.2043,10.7484)( 1.2566,10.7373)( 1.3090,10.7669)( 1.3614,10.8358)( 1.4137,10.9437)( 1.4661,11.0905)( 1.5184,11.2772)( 1.5708,11.5053)( 1.6232,11.7772)( 1.6755,12.0959)( 1.7279,12.4657)( 1.7802,12.8915)( 1.8326,13.3797)( 1.8850,13.9382)( 1.9373,14.5767)( 1.9897,15.3072)( 2.0420,16.1448)( 2.0944,17.1081)
      };
\fill ( 1.2566,10.7373*0.05) circle (0.25pt) node [below] {\tiny $1073.7$};
      \draw [yscale=0.05,smooth] plot coordinates{
( 0.5236, 6.6527)( 0.5760, 6.1477)( 0.6283, 5.7632)( 0.6807, 5.4681)( 0.7330, 5.2413)( 0.7854, 5.0685)( 0.8378, 4.9392)( 0.8901, 4.8462)( 0.9425, 4.7840)( 0.9948, 4.7487)( 1.0472, 4.7375)( 1.0996, 4.7482)( 1.1519, 4.7795)( 1.2043, 4.8305)( 1.2566, 4.9008)( 1.3090, 4.9904)( 1.3614, 5.0997)( 1.4137, 5.2293)( 1.4661, 5.3802)( 1.5184, 5.5540)( 1.5708, 5.7525)( 1.6232, 5.9780)( 1.6755, 6.2334)( 1.7279, 6.5222)( 1.7802, 6.8487)( 1.8326, 7.2180)( 1.8850, 7.6365)( 1.9373, 8.1120)( 1.9897, 8.6540)( 2.0420, 9.2745)( 2.0944, 9.9884)
      };
\fill ( 1.0472, 4.7375*0.05) circle (0.25pt);
\draw[dotted] ( pi/6,4.7375*0.05) node[left] {$48\pi^2$}-- ++(pi/6,0);
    \end{tikzpicture}
  \label{figp}
    }

    \subfloat[Area scaling: $\lambda A$]{
    \begin{tikzpicture}[scale=6,smooth]
      \draw[<->] (pi/6,1) |- (2.15,0);
      \clip (0.3,1) rectangle (2.2,-0.1);
      \draw (pi/6,0) node[below] {$\frac{\pi}{6}$};
      \draw (pi/3,0) node[below] {$\frac{\pi}{3}$};
      \draw (pi/2,0) node[below] {$\frac{\pi}{2}$};
      \draw (2*pi/3,0) node[below] {$\frac{2\pi}{3}$};
      \draw (2.18,0) node[below] {$\alpha$};
      \draw (pi/6,0) node[left] {\small $0$};
      \draw[dotted] (pi/3,0) -- (pi/3,1);
      \draw[dotted] (pi/2,0) -- (pi/2,1);
      \draw [mark=+, mark size=0.2] plot coordinates{(pi/6,0) (pi/3,0) (pi/2,0) (2*pi/3,0)};
      \draw [yscale=0.12,smooth] plot coordinates{
( 0.5236, 4.9075)( 0.5760, 4.8877)( 0.6283, 4.8884)( 0.6807, 4.9055)( 0.7330, 4.9362)( 0.7854, 4.9783)( 0.8378, 5.0303)( 0.8901, 5.0911)( 0.9425, 5.1598)( 0.9948, 5.2358)
( 1.0472, 5.3185)( 1.0996, 5.4075)( 1.1519, 5.5023)( 1.2043, 5.6027)( 1.2566, 5.7084)( 1.3090, 5.8187)( 1.3614, 5.9333)( 1.4137, 6.0514)( 1.4661, 6.1721)( 1.5184, 6.2940)( 1.5708, 6.4155)( 1.6232, 6.5346)( 1.6755, 6.6492)( 1.7279, 6.7571)( 1.7802, 6.8570)( 1.8326, 6.9488)( 1.8850, 7.0334)( 1.9373, 7.1131)( 1.9897, 7.1907)( 2.0420, 7.2694)( 2.0944, 7.3523)
     };
     \fill ( 0.5960, 4.8877*0.12) circle (0.25pt) node [below] {\tiny $48.88$};
     \draw[dotted] ( pi/6,5.3185*0.12) node[left] {$\frac{28\pi^2}{3\sqrt{3}}$}-- ++(pi/6,0);
      \draw [yscale=0.12, dashed,smooth] plot coordinates{
( 0.5236, 7.3379)( 0.5760, 6.9538)( 0.6283, 6.6358)( 0.6807, 6.3690)( 0.7330, 6.1431)( 0.7854, 5.9501)( 0.8378, 5.7845)( 0.8901, 5.6415)( 0.9425, 5.5180)( 0.9948, 5.4110)
( 1.0472, 5.3185)( 1.0996, 5.2387)( 1.1519, 5.1703)( 1.2043, 5.1121)( 1.2566, 5.0632)( 1.3090, 5.0230)( 1.3614, 4.9907)( 1.4137, 4.9661)( 1.4661, 4.9487)( 1.5184, 4.9384)( 1.5708, 4.9350)( 1.6232, 4.9384)( 1.6755, 4.9488)( 1.7279, 4.9662)( 1.7802, 4.9908)( 1.8326, 5.0230)( 1.8850, 5.0633)( 1.9373, 5.1122)( 1.9897, 5.1705)( 2.0420, 5.2389)( 2.0944, 5.3187)
      };
     \fill ( 1.5708, 4.9350*0.12) circle (0.25pt) node [below] {\tiny $5\pi^2\approx 49.35$};
      \draw [yscale=0.12,smooth] plot coordinates{
( 0.5236, 2.6239)( 0.5760, 2.5386)( 0.6283, 2.4712)( 0.6807, 2.4178)( 0.7330, 2.3759)( 0.7854, 2.3433)( 0.8378, 2.3186)( 0.8901, 2.3006)( 0.9425, 2.2884)( 0.9948, 2.2815)( 1.0472, 2.2793)( 1.0996, 2.2814)( 1.1519, 2.2875)( 1.2043, 2.2974)( 1.2566, 2.3110)( 1.3090, 2.3281)( 1.3614, 2.3488)( 1.4137, 2.3730)( 1.4661, 2.4007)( 1.5184, 2.4322)( 1.5708, 2.4674)( 1.6232, 2.5067)( 1.6755, 2.5503)( 1.7279, 2.5984)( 1.7802, 2.6514)( 1.8326, 2.7098)( 1.8850, 2.7741)( 1.9373, 2.8450)( 1.9897, 2.9231)( 2.0420, 3.0095)( 2.0944, 3.1053)
      };
     \fill ( 1.0472, 2.2793*0.12) circle (0.25pt);
     \draw[dotted] ( pi/6,2.2793*0.12) node[left] {$\frac{4\pi^2}{\sqrt{3}}$}-- ++(pi/6,0);
    \end{tikzpicture}
  \label{figa}
    }
  \end{center}
  \caption{Plots of the fundamental tone, the smallest antisymmetric tone (dashed), and the smallest symmetric tone larger than the fundamental tone, for the isosceles triangle $T(\alpha)$ of aperture $\alpha$. Global minimum points are indicated with dots.}
  \label{figLA}
\end{figure}

The eigenvalues were computed by the PDE Toolbox in Matlab, using the finite element method with about 1 million triangles. To ensure good precision we have avoided degenerate cases, restricting to $\pi/6 < \alpha < 2\pi/3$.

The Figures suggest several monotonicity conjectures. We will prove two of them, for the fundamental tone and the lowest antisymmetric eigenvalue.
\begin{proposition}[Fundamental tone] \ \label{prop:fund}

\noindent (i). $\lambda_1(\alpha) l^2$ is strictly decreasing for $0<\alpha \leq \pi/3$ and strictly increasing for $\pi/2 \leq \alpha < \pi$, and it tends to infinity at the endpoints $\alpha=0,\pi$.

\noindent (ii). $\lambda_1(\alpha) D^2$ is strictly decreasing for $0<\alpha \leq \pi/3$ and strictly increasing for $\pi/3 \leq \alpha < \pi$. The same is true under perimeter and area scalings (see Figure~\ref{figLA}), except that we do not claim ``strictness'' under area scaling.
\end{proposition}
According to Figure \ref{figl}, $\lambda_1(\alpha)l^2$ has a local minimum somewhere between $\pi/3$ and $\pi/2$. Numerically, the minimum occurs at $\alpha \simeq 1.3614$ with value $\lambda_1 l^2\simeq 48.03$. Rigorously, one can prove a weaker fact, that the minimum does not occur at $\alpha=\pi/3$ or $\pi/2$, by applying P\'{o}lya's upper bound \eqref{polyaup} to the superequilateral triangle $T(\alpha)$ with $\alpha=5\pi/12$.
\begin{proposition}[Antisymmetric tone] \ \label{prop:antisymm}

\noindent (i). $\lambda_a(\alpha) l^2$ is strictly decreasing for $0<\alpha \leq \pi/2$ and strictly increasing for $\pi/2 \leq \alpha < \pi$.

\noindent (ii). $\lambda_a(\alpha) A$ is decreasing for $0<\alpha \leq \pi/2$ and increasing for $\pi/2 \leq \alpha < \pi$.

\noindent (iii). $\lambda_a(\alpha) D^2$ is strictly decreasing for $0<\alpha \leq \pi/3$ and strictly increasing for $\pi/3 \leq \alpha < \pi$.
\end{proposition}
\begin{proof}[Proof of Proposition \ref{prop:fund}]
The isosceles triangle $T(\alpha)$ has area $A=\frac{1}{2}l^2 \sin \alpha$, and diameter
\[
D =
\begin{cases}
l , & 0 < \alpha \leq \frac{\pi}{3} , \\
2l \sin(\alpha/2) , & \frac{\pi}{3} \leq \alpha < \pi ,
\end{cases}
\]
and perimeter $L=2l(1+\sin \big( \alpha/2) \big)$.

Part (ii). It was proved by Siudeja \cite[Theorem 1.3]{S10} using continuous symmetrization that $\lambda_1(\alpha)A$ is decreasing for $\alpha\in(0,\pi/3]$ and increasing for $\alpha\in[\pi/3,\pi)$.

Strict monotonicity for $\lambda_1(\alpha)L^2$ and $\lambda_1(\alpha)D^2$ then follows from strict monotonicity of $L^2/A$ and $D^2/A$, when $\alpha \leq \pi/3$ and when $\alpha \geq \pi/3$.

Part (i). For $\alpha\le \pi/3$ the area is strictly increasing while for $\alpha \geq \pi/2$ it is strictly decreasing. Hence $\lambda_1(\alpha)l^2$ is strictly decreasing for $\alpha \leq \pi/3$ and strictly increasing for $\alpha \geq \pi/2$.
\end{proof}

\begin{proof}[Proof of Proposition \ref{prop:antisymm}]
Part (ii). First, note that $T(\alpha)$ and $T(\pi-\alpha)$ have the same area $A$ and the same antisymmetric tone, $\lambda_a(\alpha)=\lambda_a(\pi-\alpha)$, as indicated in Figure \ref{figsame}.

Assume $\alpha<\pi/2$. The upper part of the acute isosceles triangle $T(\alpha)$ is a right triangle $R(\alpha/2)$. The fundamental tone of $R(\alpha/2)$ equals $\lambda_a(\alpha)$. According to Lemma~\ref{lemmabelow} below, the fundamental tone of $R(\alpha/2)$ is decreasing with $\alpha\in (0,\pi/2]$, when normalized by area. Therefore $\lambda_a(\alpha) A$ is decreasing for $\alpha\in(0,\pi/2]$, and hence is increasing for $\alpha\in[\pi/2,\pi)$ (by replacing $\alpha$ with $\pi-\alpha$).

  \begin{figure}[t]
    \begin{center}
      \begin{tikzpicture}[scale=3]
	\draw (0,0) -- (2,0) -- (1,1/2) -- (1,-1/2) -- (0,0) -- (1,1/2);
	\draw (3/4,1/5) node {$+$};
	\draw (3/4,-1/5) node {$-$};
	\draw (5/4,1/5) node {$-$};
	\draw (1/6,0) node {$\alpha$};
	\draw (1,9/24) node {$\pi-\alpha$};
	\clip (0,0) -- (1,1/2) -- (2,0) -- (1,0) -- (1,-1/2) -- cycle;
	\draw (0,0) circle (1/4);
	\draw (1,1/2) circle (1/4);
      \end{tikzpicture}
    \end{center}
    \caption{Nodal domains for antisymmetric eigenfunctions of $T(\alpha)$ and $T(\pi-\alpha)$.}
    \label{figsame}
  \end{figure}
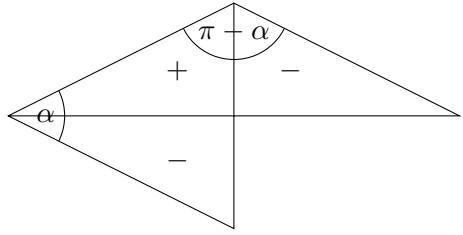

Part (i). Since the area of $T(\alpha)$ is strictly increasing for $\alpha \leq \pi/2$ and strictly decreasing for $\alpha \geq \pi/2$, we deduce that $\lambda_a(\alpha)l^2$ is strictly decreasing for $\alpha \leq \pi/2$ and strictly increasing for $\alpha \geq \pi/2$.

Part (iii). For $\alpha \leq \pi/3$ the diameter of $T(\alpha)$ is fixed, and hence $\lambda_a(\alpha)D^2$ is strictly decreasing for $\alpha\in(0,\pi/3]$.

Let $Q(a)$ be the superequilateral triangle with diameter $D$ having longest side on the interval from $(0,-D/2)$ to $(0,D/2)$ and third vertex at $(a,0)$. Then $Q(a) \varsubsetneqq Q(b)$ if $a<b \leq (\sqrt{3}/2)D$, so that $\lambda_a(a) > \lambda_a(b)$ by domain monotonicity. Hence $\lambda_a(\alpha)D(\alpha)^2$ is strictly increasing for $\alpha\in[\pi/3,\pi)$.
\end{proof}

We must still prove:

\begin{lemma} \label{lemmabelow}
  Let $R(\alpha)$ be a right triangle with smallest angle $\alpha$. Then $\left. \lambda_1 A \right|_{R(\alpha)}$ is decreasing for $\alpha\in(0,\pi/4]$.
\end{lemma}

\begin{proof}[Proof of Lemma~\ref{lemmabelow}]
Let $\alpha<\beta\le\pi/4$. Assume the hypotenuse of $R(\alpha)$ has length $1$. The other sides have lengths $\sin \alpha$ and $\cos \alpha$, and the area of $R(\alpha)$ equals $\sin(2\alpha)/4$.  Assume the hypotenus of $R(\beta)$ has length $\sqrt{\sin(2\alpha)/\sin(2\beta)}$, so that $R(\beta)$ has the same area as $R(\alpha)$.

Consider the Steiner symmetrization of $R(\alpha)$, that is, an obtuse isosceles triangle with longest side $1$ and altitude $\sin(2\alpha)/2$ (see Figure \ref{rightsym}). The shorter sides have lengths $x=\sqrt{1+\sin^2(2\alpha)}/2$.

  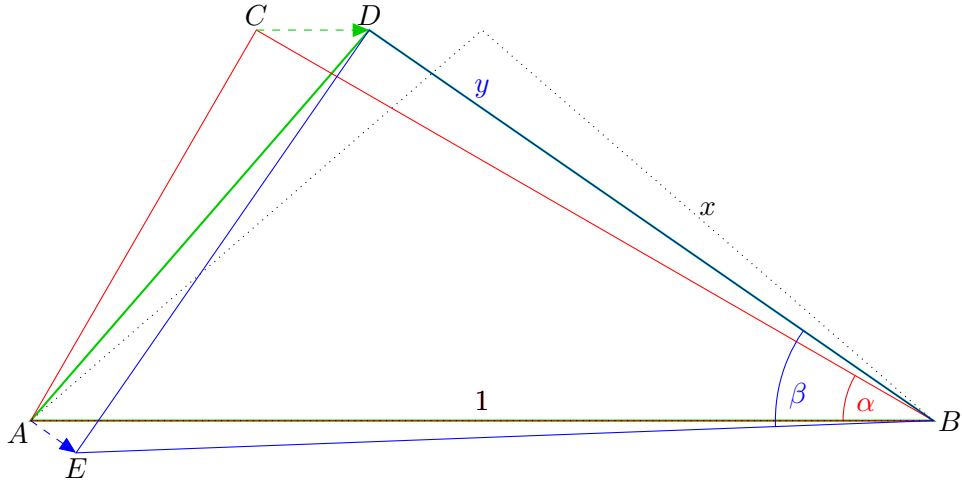
\begin{figure}[t]
    \begin{tikzpicture}[scale=3]
      \draw[thick,green!80!black] (0,0) -- (4,0) -- (1.5,{sqrt(3)}) -- cycle;
      \draw[red] (0,0) -- node [above,pos=0.5] {$1$} (4,0) -- (1,{sqrt(3)}) -- cycle;
      \begin{scope}
           \clip (0,0) -- node [above,pos=0.5] {$1$} (4,0) -- (1,{sqrt(3)}) -- cycle;
	   \draw[red] (4,0) circle (0.4);
	   \draw[red] (3.7,0.07) node {$\alpha$};
      \end{scope}
      \draw[dotted] (0,0) -- (4,0) -- node [above,pos=0.5] {$x$} (2,{sqrt(3)}) -- cycle;
      \draw[blue] (1.5,{sqrt(3)}) -- node [above,pos=0.2] {$y$} (4,0) -- ($(1.5,{sqrt(3)})!8cm*sqrt(3)/sqrt(37)!-90:(4,0)$) coordinate (a) -- cycle;
      \begin{scope}
           \clip (1.5,{sqrt(3)}) -- node [above,pos=0.2] {$y$} (4,0) -- ($(1.5,{sqrt(3)})!8cm*sqrt(3)/sqrt(37)!-90:(4,0)$) coordinate (a) -- cycle;
	   \draw[blue] (4,0) circle (0.7);
	   \draw[blue] (3.4,0.11) node {$\beta$};
      \end{scope}
      \draw[dashed,blue,-triangle 45] (0,0) -- (a);
      \draw[green!80!black,dashed,-triangle 45] (1,{sqrt(3)}) -- +(1/2,0);
      \draw (0,0) node [below left=-3pt] {$A$};
      \draw (4,0) node [right=-2pt] {$B$};
      \draw (1,{sqrt(3)}) node [above=-2pt] {$C$};
      \draw (1.5,{sqrt(3)}) node [above=-2pt] {$D$};
      \draw (a) node [below=-2pt] {$E$};
    \end{tikzpicture}
    \caption{Transformation of the right triangle $R(\alpha)=ABC$ onto $R(\beta)=BDE$. The triangle $ABD$ is the first step of the continuous symmetrization performed along arrow $CD$. The second symmetrization is performed along arrow $AE$. }
    \label{rightsym}
  \end{figure}

Suppose $\beta$ satisfies $\sin\alpha<\sin\beta\le \sqrt{2}\sin\alpha\cos\alpha$. This assumption ensures (after a short calculation) that $x<y$ where
\[
y=\cos \beta \sqrt{\sin(2\alpha)/\sin(2\beta)} ,
\]
is the length of the longer leg of $R(\beta)$.

We perform a continuous Steiner symmetrization with respect to the line perpendicular to the hypotenuse of $R(\alpha)$. We stop when the longer of the moving sides has length equal to $y$. The resulting triangle is obtuse. Now perform another continuous symmetrization with respect to the line perpendicular to the side of length $y$. Stop when the triangle becomes a right triangle. This right triangle has the same area as $R(\beta)$ and has one leg of the same length $y$; hence it must be $R(\beta)$. Its fundamental tone is less than or equal to that of $R(\alpha)$, by properties of continuous symmetrization. Hence $\left. \lambda_1 A \right|_{R(\alpha)} \geq \left. \lambda_1 A \right|_{R(\beta)}$. It follows that
$\left. \lambda_1 A \right|_{R(\alpha)}$ is decreasing for $\alpha \in (0,\pi/4]$.
\end{proof}

\begin{remark} \rm
For the second eigenvalue, Figure~\ref{figLA} shows that $\lambda_2 A$ and $\lambda_2 L^2$ are \emph{not} minimal among isosceles triangles at the equilateral triangle. This fact is analogous to the situation for convex domains, where the minimizers are certain ``stadium-like'' sets rather than disks (Henrot \emph{et al.} \cite{BBH09,HO03}).

\end{remark}

\appendix
\section{\bf The equilateral triangle and its eigenvalues} \label{equilateral}

The frequencies of the equilateral triangle were derived
about 150 years ago by Lam\'{e} \cite[pp.~131--135]{L66}. See the treatment in the text of Mathews and Walker \cite[pp.~237--239]{MW70} or in the paper by Pinsky \cite{Pi85}. Note also the recent exposition by McCartin \cite{M03}, which builds on work of Pr\'{a}ger \cite{Pr98}.

The equilateral triangle $E_1$ with sidelength $1$ has eigenvalues forming a doubly-indexed sequence:
\[
\sigma_{m,n} = (m^2 + mn + n^2) \cdot \frac{16 \pi^2}{9} , \qquad m,n \geq 1 .
\]
The first three eigenvalues are
\[
\lambda_1 = 3 \cdot \frac{16 \pi^2}{9} = \sigma_{1,1} \qquad \text{and} \qquad \lambda_2 = \lambda_3 =7 \cdot \frac{16 \pi^2}{9} = \sigma_{1,2}=\sigma_{2,1} .
\]
%

Indices with $m<n$ correspond to antisymmetric eigenfunctions. We denote those antisymmetric eigenvalues by $\lambda_1^a \leq \lambda_2^a \leq \cdots$.

\begin{lemma} \label{explicit}
We have $\lambda_j^a > \frac{11}{6} \lambda_j$ for $j=1,2,3$ and for $j=5,6,\ldots,110$. For the exceptional case $j=4$ where $\lambda_4^a < \frac{11}{6} \lambda_4$, one still has
\[
(\lambda_1^a+\lambda_2^a+\lambda_3^a+\lambda_4^a) > \frac{11}{6} (\lambda_1+\lambda_2+\lambda_3+\lambda_4) .
\]
\end{lemma}
To prove the lemma, simply compute the  first $110$ eigenvalues $\lambda_j$ and antisymmetric eigenvalues $\lambda_j^a$, using the indices $m$ and $n$ listed in Table~\ref{tab:}.

Lemma~\ref{explicit} is used to help prove Lemma~\ref{compequilateral}, in Section~\ref{seccomp}.

\newcommand{\binomb}[2]{\genfrac{[}{]}{0pt}{}{#1}{#2}}
\newcommand{\frace}[2]{\genfrac{}{}{0pt}{}{#2}{#1}}
\newcommand\T{\rule{0pt}{2.8ex}}
\newcommand\B{\rule[-1.5ex]{0pt}{0pt}}

{
\begin{table}
  \centering
  \begin{tabular}{|c|c|c|c|c|c|c|c|c|c|c|c|c|c|}
    \hline
\T\B$\!\!\frace{(1,1)}{[1,2]}\!\!$&$ \!\!\frace{(2,1)}{[1,3]}\!\!$&$
\!\!\frace{(1,2)}{[2,3]}\!\!$&$\!\!\frace{(2,2)}{[1,4]}\!\!$&$ \!\!\frace{(3,1)}{[2,4]}\!\!$&$
\!\!\frace{(1,3)}{[1,5]}\!\!$&$ \!\!\frace{(3,2)}{[3,4]}\!\!$&$ \!\!\frace{(2,3)}{[2,5]}\!\!$&$ \!\!\frace{(4,1)}{[1,6]}\!\!$&$ \!\!\frace{(1,4)}{[3,5]}\!\!$\\
    \hline
\T\B$\!\!\frace{(3,3)}{[2,6]}\!\!$&$ \!\!\frace{(4,2)}{[1,7]}\!\!$&$ \!\!\frace{(2,4)}{[4,5]}\!\!$&$\!\!\frace{(5,1)}{[3,6]}\!\!$&$ \!\!\frace{(1,5)}{[2,7]}\!\!$&$ \!\!\frace{(4,3)}{[1,8]}\!\!$&$ \!\!\frace{(3,4)}{[4,6]}\!\!$&$ \!\!\frace{(5,2)}{[3,7]}\!\!$&$ \!\!\frace{(2,5)}{[2,8]}\!\!$&$ \!\!\frace{(6,1)}{[5,6]}\!\!$\\
    \hline
\T\B$\!\!\frace{(1,6)}{[1,9]}\!\!$&$ \!\!\frace{(4,4)}{[4,7]}\!\!$&$ \!\!\frace{(5,3)}{[3,8]}\!\!$&$ \!\!\frace{(3,5)}{[2,9]}\!\!$&$ \!\!\frace{(6,2)}{[5,7]}\!\!$&$ \!\!\frace{(2,6)}{[1,10]}\!\!$&$ \!\!\frace{(7,1)}{[4,8]}\!\!$&$ \!\!\frace{(1,7)}{[3,9]}\!\!$&$ \!\!\frace{(5,4)}{[2,10]}\!\!$&$ \!\!\frace{(4,5)}{[6,7]}\!\!$\\
    \hline
\T\B$\!\!\frace{(6,3)}{[5,8]}\!\!$&$ \!\!\frace{(3,6)}{[4,9]}\!\!$&$ \!\!\frace{(7,2)}{[1,11]}\!\!$&$ \!\!\frace{(2,7)}{[3,10]}\!\!$&$ \!\!\frace{(8,1)}{[2,11]}\!\!$&$ \!\!\frace{(1,8)}{[6,8]}\!\!$&$ \!\!\frace{(5,5)}{[5,9]}\!\!$&$ \!\!\frace{(6,4)}{[4,10]}\!\!$&$ \!\!\frace{(4,6)}{[1,12]}\!\!$&$ \!\!\frace{(7,3)}{[3,11]}\!\!$\\
    \hline
\T\B$\!\!\frace{(3,7)}{[7,8]}\!\!$&$ \!\!\frace{(8,2)}{[6,9]}\!\!$&$ \!\!\frace{(2,8)}{[2,12]}\!\!$&$ \!\!\frace{(9,1)}{[5,10]}\!\!$&$ \!\!\frace{(6,5)}{[4,11]}\!\!$&$ \!\!\frace{(5,6)}{[1,13]}\!\!$&$ \!\!\frace{(1,9)}{[3,12]}\!\!$&$ \!\!\frace{(7,4)}{[7,9]}\!\!$&$ \!\!\frace{(4,7)}{[6,10]}\!\!$&$ \!\!\frace{(8,3)}{[2,13]}\!\!$\\
    \hline
\T\B$\!\!\frace{(3,8)}{[5,11]}\!\!$&$ \!\!\frace{(9,2)}{[4,12]}\!\!$&$ \!\!\frace{(2,9)}{[1,14]}\!\!$&$ \!\!\frace{(6,6)}{[8,9]}\!\!$&$ \!\!\frace{(7,5)}{[3,13]}\!\!$&$ \!\!\frace{(5,7)}{[7,10]}\!\!$&$ \!\!\frace{(10,1)}{[6,11]}\!\!$&$ \!\!\frace{(1,10)}{[2,14]}\!\!$&$ \!\!\frace{(8,4)}{[5,12]}\!\!$&$ \!\!\frace{(4,8)}{[4,13]}\!\!$\\
    \hline
\T\B$\!\!\frace{(9,3)}{[1,15]}\!\!$&$ \!\!\frace{(3,9)}{[8,10]}\!\!$&$ \!\!\frace{(10,2)}{[7,11]}\!\!$&$ \!\!\frace{(2,10)}{[3,14]}\!\!$&$ \!\!\frace{(7,6)}{[6,12]}\!\!$&$ \!\!\frace{(6,7)}{[5,13]}\!\!$&$ \!\!\frace{(8,5)}{[2,15]}\!\!$&$ \!\!\frace{(5,8)}{[4,14]}\!\!$&$ \!\!\frace{(11,1)}{[9,10]}\!\!$&$ \!\!\frace{(9,4)}{[8,11]}\!\!$\\
    \hline
\T\B$\!\!\frace{(4,9)}{[1,16]}\!\!$&$ \!\!\frace{(1,11)}{[7,12]}\!\!$&$ \!\!\frace{(10,3)}{[3,15]}\!\!$&$ \!\!\frace{(3,10)}{[6,13]}\!\!$&$ \!\!\frace{(11,2)}{[5,14]}\!\!$&$ \!\!\frace{(7,7)}{[2,16]}\!\!$&$ \!\!\frace{(2,11)}{[9,11]}\!\!$&$ \!\!\frace{(8,6)}{[4,15]}\!\!$&$ \!\!\frace{(6,8)}{[8,12]}\!\!$&$ \!\!\frace{(9,5)}{[1,17]}\!\!$\\
    \hline
\T\B$\!\!\frace{(5,9)}{[7,13]}\!\!$&$ \!\!\frace{(10,4)}{[3,16]}\!\!$&$ \!\!\frace{(4,10)}{[6,14]}\!\!$&$ \!\!\frace{(12,1)}{[5,15]}\!\!$&$ \!\!\frace{(1,12)}{[2,17]}\!\!$&$ \!\!\frace{(11,3)}{[10,11]}\!\!$&$ \!\!\frace{(3,11)}{[9,12]}\!\!$&$ \!\!\frace{(8,7)}{[4,16]}\!\!$&$ \!\!\frace{(7,8)}{[8,13]}\!\!$&$ \!\!\frace{(9,6)}{[7,14]}\!\!$\\
    \hline
\T\B$\!\!\frace{(6,9)}{[1,18]}\!\!$&$ \!\!\frace{(12,2)}{[3,17]}\!\!$&$ \!\!\frace{(2,12)}{[6,15]}\!\!$&$ \!\!\frace{(10,5)}{[5,16]}\!\!$&$ \!\!\frace{(5,10)}{[10,12]}\!\!$&$ \!\!\frace{(11,4)}{[2,18]}\!\!$&$ \!\!\frace{(4,11)}{[9,13]}\!\!$&$ \!\!\frace{(13,1)}{[8,14]}\!\!$&$ \!\!\frace{(1,13)}{[4,17]}\!\!$&$ \!\!\frace{(12,3)}{[7,15]}\!\!$\\
    \hline
\T\B$\!\!\frace{(3,12)}{[1,19]}\!\!$&$ \!\!\frace{(8,8)}{[3,18]}\!\!$&$ \!\!\frace{(9,7)}{[6,16]}\!\!$&$ \!\!\frace{(7,9)}{[11,12]}\!\!$&$ \!\!\frace{(10,6)}{[10,13]}\!\!$&$ \!\!\frace{(6,10)}{[5,17]}\!\!$&$ \!\!\frace{(13,2)}{[9,14]}\!\!$&$ \!\!\frace{(2,13)}{[2,19]}\!\!$&$ \!\!\frace{(11,5)}{[8,15]}\!\!$&$ \!\!\frace{(5,11)}{[4,18]}\!\!$\\
    \hline
  \end{tabular}
  \caption{Pairs of integers $(m,n)$ giving the first 110 eigenvalues $\lambda_j$ along with pairs $[m,n]$ giving the first 110 antisymmetric eigenvalues $\lambda_j^a$, for an equilateral triangle. The index $j$ increases from $1$ to $10$ across the first row, and so on. 
  \vspace*{-24pt}
  }
  \label{tab:}
\end{table}
}

\end{document}